\documentclass[10pt,reqno]{amsart}

\usepackage[backref=page]{hyperref}
\usepackage[usenames,dvipsnames]{color}
\usepackage{amssymb}
\usepackage{amsmath}
\usepackage{amsfonts}
\usepackage{graphicx,cite}
\usepackage{latexsym}
\usepackage{mathdots}
\usepackage{enumerate}
\usepackage{lscape}
\usepackage[table]{xcolor}

\newtheorem{Theorem}{Theorem}[section]
\newtheorem{Lemma}[Theorem]{Lemma}
\newtheorem{Corollary}[Theorem]{Corollary}

\numberwithin{Theorem}{section}

\theoremstyle{definition}
\newtheorem{Example}{Example}

\newtheorem*{Definition}{Definition}

\numberwithin{equation}{section}

\newcommand{\inner}[1]{\left< #1 \right>}

\newcommand{\0}{\textcolor{gray}{0}}
\newcommand{\gcdots}{\textcolor{gray}{\cdots}}
\newcommand{\gvdots}{\textcolor{gray}{\vdots}}
\newcommand{\gddots}{\textcolor{gray}{\ddots}}

\theoremstyle{definition}

\newcommand{\baseRing}[1]{\ensuremath{\mathbb{#1}}}
\newcommand{\C}{\baseRing{C}}

\newcommand{\Z}{\mathbb{Z}}

\newcommand{\K}{\mathcal{K}}
\newcommand{\X}{\mathcal{X}}
\renewcommand{\vec}[1]{{\bf #1}}
\newcommand{\Irr}{\operatorname{Irr}}
\newcommand{\diag}{\operatorname{diag}}
\newcommand{\Aut}{\operatorname{Aut}}
\newcommand{\tr}{\operatorname{tr}}

\newcommand{\norm}[1]{\|#1\|}
\newcommand{\minimatrix}[4]{\begin{bmatrix}#1&#2\\#3&#4\end{bmatrix}}

\allowdisplaybreaks

\begin{document}
\title{Ramanujan sums as supercharacters}

\author[C.F.~Fowler]{Christopher F. Fowler}
\address{University of Washington\\
Department of Mathematics\\
Box 354350\\
Seattle, WA 98195-4350}
\email{cff2008@math.washington.edu}

\author[S.R.~Garcia]{Stephan Ramon Garcia}
    \address{   Department of Mathematics\\
            Pomona College\\
            Claremont, California\\
            91711 \\ USA}
    \email{Stephan.Garcia@pomona.edu}
    \urladdr{\url{http://pages.pomona.edu/~sg064747}}

\author[G.~Karaali]{Gizem Karaali}
    \email{Gizem.Karaali@pomona.edu}
    \urladdr{\url{http://pages.pomona.edu/~gk014747}}

\thanks{S.R.~Garcia partially funded by NSF grant DMS-1001614. G.~Karaali partially funded by a NSA Young Investigator Award. }

\begin{abstract}
	The theory of \emph{supercharacters}, recently developed by Diaconis-Isaacs
	and Andr\'e, can be used to derive the fundamental algebraic properties of Ramanujan sums.
	This machinery frequently yields one-line proofs of difficult identities and 
	provides many novel formulas.  In addition to exhibiting a new application of supercharacter theory, 
	this article also serves as a blueprint for future work since
	some of the abstract results we develop are applicable in much greater generality.
\end{abstract}


\maketitle

\bibliographystyle{plain}


\section{Introduction}

	Our primary aim in this note is to demonstrate that most of the fundamental algebraic properties of Ramanujan sums
	can be deduced using the theory of \emph{supercharacters}, recently developed by Diaconis-Isaacs
	and Andr\'e.  In fact, the machinery of supercharacter theory frequently yields one-line proofs of many
	difficult identities and provides an array of new tools which can be used to derive various novel formulas.
	Our approach is entirely systematic, relying on a flexible and general framework.
	Indeed, we hope to convince the reader that supercharacter theory provides a natural 
	framework for the study of Ramanujan sums.
	In addition to exhibiting a novel application of supercharacter theory, this article also serves as a blueprint for future work since
	some of the abstract results which we develop are applicable in much greater generality (see \cite{SESUP}).
	
\subsection{Ramanujan sums}
	In what follows, we let $e(x) = \exp(2 \pi i x)$, so that the function $e(x)$ is periodic with period $1$.
	For integers $n,x$ with $n \geq 1$, the expression
	\begin{equation}\label{eq-RamanujanSumDefinition}
		c_n(x) = \sum_{ \substack{ j = 1 \\ (j,n) = 1} }^n e\left( \frac{jx}{n} \right)
	\end{equation}
	is called a \emph{Ramanujan sum} (or sometimes \emph{Ramanujan's sum}).
	Ramanujan himself (1918) \cite[Paper 21]{Ramanujan} noted that 
	Dirichlet and Dedekind had already considered such expressions
	in their famed text \emph{Vorlesungen \"uber Zahlentheorie} (1863).  Moreover,
	certain related identities were already known to von Sterneck (1902) \cite{vonSterneck},
	Kluyver  (1906) \cite{Kluyver}, Landau (1909) \cite{Landau}, and Jensen (1915) \cite{Jensen}.
	Nevertheless, ``Ramanujan was the first to appreciate the importance of the sum and to use it systematically,''
	according to G.H.~Hardy \cite[p.~159]{HardyRamanujan}.

	Ramanujan's interest in the sums \eqref{eq-RamanujanSumDefinition} originated in his
	desire to ``obtain expressions for a variety of well-known arithmetical functions of $n$ in the form of
	a series $\sum_s a_s c_s(n)$.''  This particular analytic aspect of the subject has flourished in the intervening years and is discussed at length
	in \cite{Lucht,SchwarzSurvey,SchwarzBook, McCarthyBook}.
	On the other hand, in classical character theory Ramanujan sums can be used to 
	establish the integrality of the character values for the symmetric group \cite[Cor.~22.17]{JamesLiebeck}.
	However, perhaps the most famous appearance of Ramanujan sums is their crucial role
	in Vinogradov's proof that every sufficiently large odd number is the sum of three primes \cite[Ch.~8]{Nathanson}.

	In more recent years, Ramanujan sums have appeared in the study of	
	Waring-type formulas \cite{Konvalina}, 
	the distribution of rational numbers in short intervals \cite{Jutila}, 
	equirepartition modulo odd integers \cite{Balandraud},
	the large sieve inequality \cite{Ramare},
	graph theory \cite{Droll},
	symmetry classes of tensors \cite{TYT},
	combinatorics \cite{Rao},
	cyclotomic polynomials \cite{TothCyclotomic, Motose, Nicol, ErdosVaughan},
	and Mahler matrices \cite{LehmerMahler}.
	In physics, Ramanujan sums have applications in the processing of low-frequency noise \cite{Planat2} and of long-period sequences \cite{Planat3}
	and in the study of quantum phase locking \cite{Planat}.	
	We should also remark that various generalizations of the classical Ramanujan sum \eqref{eq-RamanujanSumDefinition}
	have arisen over the years \cite{AndersonApostol, Cohen, CohenMAA, Sugunamma}
	and that Ramanujan sums involving matrix variables have also been considered \cite{Nanda, Ramanathan}.

\subsection{Supercharacters}
	The theory of supercharacters, of which classical character theory is a special case, was recently introduced 
	by P.~Diaconis and I.M.~Isaacs (2008) \cite{DiIs08} to generalize the \emph{basic characters} of 
	C.~Andr\'e \cite{An95, An01, An02}.  	
	Here we summarize a few important facts.  Further details can be found in \cite{DiIs08, Hendrickson}.

	\begin{Definition}[Diaconis-Isaacs \cite{DiIs08}]
		Let $G$ be a finite group, let $\K$ be a partition of $G$, and let $\X$ be a partition of the set $\Irr(G)$ 
		of irreducible characters of $G$. We call the ordered pair $(\X, \K)$ a \emph{supercharacter theory} if 
		\begin{enumerate}\addtolength{\itemsep}{0.5\baselineskip}
			\item $\{1\} \in \K$
			\item $|\X| = |\K|$
			\item For each $X \in \X$ the character  
				\begin{equation*}
					\sigma_X = \sum_{\chi \in X} \chi(1)\chi
				\end{equation*}
				is constant on each $K \in \K$.
		\end{enumerate}
		The characters $\sigma_X$ are called \emph{supercharacters} and the elements $K$ of $\K$
		 are called \emph{superclasses}.
	\end{Definition}

	As \cite[Lem.~2.1]{DiIs08} shows, the preceding definition is equivalent to 
	the following.

	\begin{Definition}[Andr\'e \cite{An09}]
		Let $G$ be a finite group, let $\K$ be a partition of $G$, and let $\X$ be a collection of complex characters of 
		$G$. We call the ordered pair $(\X, \K)$ a \emph{supercharacter theory} if 
		\begin{enumerate}\addtolength{\itemsep}{0.5\baselineskip}
			\item Every irreducible character of $G$ is a constituent of a unique $\chi \in \X$
			\item $|\X| = |\K|$
			\item Each character $\chi \in \X$ is constant on $K$ for each $K \in \K$.
		\end{enumerate}
		The elements $\chi$ of $\X$ are called \emph{supercharacters} and the sets $K$ are called \emph{superclasses}. 
	\end{Definition}
	
	Regardless of which definition one chooses to work with, it is straightforward to verify that 
	each $K$ in $\K$ is a union of conjugacy classes of $G$ and that each of the partitions $\K$ and $\X$ 
	determines the other.  The only significant difference between these two definitions is that the second approach can
	yield supercharacters which are multiples of the $\sigma_X$ defined above.
	
	In the literature to date, the main use of supercharacter theory has been to 
	perform computations when a complete character theory is difficult 
	or impossible to determine.  For instance, Andr\'e developed a successful 
	supercharacter theory for the unipotent matrix groups $U_n(q)$
	whose representation theories are known to be \emph{wild} (see also \cite{Ya01,Ya10}). 
	Supercharacter theories have also proven to be 
	relevant outside the realm of finite group theory.  For instance, 
	these notions can be used to obtain a more general 
	theory of spherical functions and Gelfand pairs \cite{DiIs08}.  In a different direction, 
	recent work has revealed deep connections between supercharacter theory and 
	the Hopf algebra of symmetric functions of noncommuting variables \cite{Aguiar}.
	Another application may be found 
	in \cite{ACDiSt04}, where the authors use supercharacter theory to study random walks on upper triangular matrices.
	Other recent work on supercharacters concerns connections with Schur rings \cite{Hendrickson,HuJo08}
	and with their combinatorial properties \cite{DiTh09, ThVe09, Th10}.
	We should also remark that similar constructions surfaced independently in the study of
	quasigroups and association schemes in the form of \emph{fusions} of character 
	tables \cite{JoSm89, JoPo99,HuJo08}.
	
\subsection{General approach}
	We proceed along a different course, turning our attention to the group
	$\Z/n\Z$, whose classical representation theory is already well-understood.   It turns out that a natural
	supercharacter theory for $\Z/n\Z$ can be developed for which Ramanujan sums appear as values of the corresponding
	supercharacters.  In this manner, the modern machinery of supercharacter theory can be used to
	generate a wide variety of formulas and identities for Ramanujan sums.  Along the way, we also develop
	a notion of \emph{superclass arithmetic} which generalizes the standard arithmetic of conjugacy classes
	from classical character theory.

\section{A supercharacter theory for $\Z/n\Z$}\label{SectionZnZ}
	In this section we introduce a supercharacter theory for $\Z/n\Z$ which arises naturally
	from the action of $\Aut(\Z/n\Z)$ on $\Z/n\Z$ (see \cite{HendricksonNew} for a description of all possible
	supercharacter theories on $\Z/n\Z$).  Before proceeding, we require a few preliminaries.  
	Although some of this material is well-known in certain circles, 
	we include the details since the theory of supercharacters 
	is not yet common knowledge among the general mathematics community.  Moreover, 
	for many of the upcoming applications we require a particular unitary rescaling of our supercharacter 
	tables which is not widely used.

\subsection{Supercharacter tables}
	Suppose that $G$ is a finite group of order $|G|$ and that
	$(\X,\K)$ is a supercharacter theory for $G$.  In other words, suppose that we have a partition
	$\X = \{X_1,X_2,\ldots,X_r\}$ of $\Irr(G)$ with corresponding supercharacters 
	\begin{equation}\label{eq-SCG}
		\sigma_i = \sum_{\chi \in X_i} \chi(1) \chi
	\end{equation}
	and a compatible partition $\K = \{K_1,K_2,\ldots,K_r\}$ of $G$ into superclasses.  
	The \emph{supercharacter table} for $G$ corresponding to $(\X,\K)$ is the $r \times r$ array
	\begin{equation}\label{eq-SuperCharacterTableGeneral}
		\begin{array}{|c||cccc|}
		\hline
		& K_1 & K_2 & \cdots & K_r \\
		\hline\hline
		\sigma_1 & \sigma_1(K_1) & \sigma_1(K_2) & \cdots & \sigma_1(K_r) \\
		\sigma_2 & \sigma_2(K_1) & \sigma_2(K_2) & \cdots & \sigma_2(K_r) \\
		\vdots & \vdots & \vdots & \ddots & \vdots \\
		\sigma_r & \sigma_r(K_1) & \sigma_r(K_2) & \cdots & \sigma_r(K_r) \\
		\hline
		\end{array}
	\end{equation}
	whose $(i,j)$ entry is $\sigma_i(K_j)$.  We let 
	\begin{equation*}
		S = \big( \sigma_i(K_j) \big)_{i,j=1}^r
	\end{equation*}
	denote the $r \times r$ matrix which encodes the data in 
	\eqref{eq-SuperCharacterTableGeneral}.  In what follows we frequently identify supercharacter
	tables with their matrix representations and we often refer to the matrix $S$ itself as
	a supercharacter table.
	
	Recall that a function $f:G\to\C$ is called a \emph{class function} if $f$ is constant
	on each conjugacy class of $G$.  The space of complex-valued class functions on $G$ is endowed
	with a natural inner product
	\begin{equation}\label{eq-StandardInnerProduct}
		\inner{\chi,\chi'} = \frac{1}{|G|} \sum_{g \in G} \chi(g)\overline{\chi'(g)},
	\end{equation}
	with respect to which the irreducible characters of $G$ form an orthonormal basis.
	In light of the fact that supercharacters are constant on superclasses 
	(i.e., they are \emph{superclass functions}), \eqref{eq-StandardInnerProduct} implies that
	\begin{equation*}
		\inner{\sigma_i,\sigma_j} 
		= \frac{1}{|G|} \sum_{\ell=1}^r |K_{\ell}| \sigma_i(K_{\ell})\overline{\sigma_j(K_{\ell})}.
	\end{equation*}
	It follows from \eqref{eq-SCG} and the orthogonality of irreducible characters that
	\begin{equation*}
		\inner{ \sigma_i, \sigma_j} = \Big<\sum_{\chi \in X_i} \chi(1) \chi ,  \sum_{\chi' \in X_j} \chi'(1) \chi' \Big>
			= \delta_{i,j} \norm{X_i}_2^2 
	\end{equation*}
	where
	\begin{equation*}
		 \norm{X_i}_2 = \sqrt{\sum_{\chi \in X_i} | \chi(1)|^2} 
	\end{equation*}
	is a convenient shorthand.
	Putting this all together, we see that
	\begin{equation}\label{eq-RowOrthogonality}
		\frac{1}{|G|}\sum_{\ell=1}^r |K_{\ell}| \sigma_i(K_{\ell})\overline{\sigma_j(K_{\ell})} = \delta_{i,j}\norm{X_i}_2^2.
	\end{equation}
	
	It is helpful to interpret the preceding result matricially.  Letting 
	\begin{equation}\label{eq-RightMultiplier}
		R = \diag( \sqrt{ |K_1|}, \sqrt{ |K_2|},\ldots, \sqrt{ |K_r| } ),
	\end{equation}
	we see that \eqref{eq-RowOrthogonality} is equivalent to asserting that
	\begin{equation*}
		(SR)(SR)^* = |G| \diag( \norm{X_1}_2^2, \norm{X_2}_2^2, \ldots, \norm{X_r}_2^2).
	\end{equation*}
	Letting
	\begin{equation}\label{eq-LeftMultiplier}
		L = \frac{1}{\sqrt{|G|}} \diag( \norm{X_1}_2^{-1}, \norm{X_2}_2^{-1},\ldots, \norm{X_r}_2^{-1}),
	\end{equation}
	we conclude that the matrix 
	\begin{equation}\label{eq-LSR}
		U = L S R
	\end{equation}
	satisfies $UU^* = I$.  In other words, the $r \times r$ matrix	
	\begin{equation}\label{eq-UnitaryMatrix}
		U = \frac{1}{\sqrt{|G|}} \left[  \frac{\sigma_i(K_j)\sqrt{  |K_j| }}{ \norm{X_i}_2} \right]_{i,j=1}^r
	\end{equation}
	is unitary.  Since $U^*U = I$, we now obtain the column orthogonality relation
	\begin{equation}\label{eq-ColumnOrthogonality}
		\frac{\sqrt{|K_i||K_j|}}{|G|} \sum_{\ell=1}^r \frac{\sigma_{\ell}(K_i)\overline{\sigma_{\ell}(K_j)}}{\norm{X_{\ell}}_2^2}
		= \delta_{i,j}.
	\end{equation}
	Similarly, we see that the equation $UU^* = I$ encodes \eqref{eq-RowOrthogonality}.

\subsection{A supercharacter table for $\Z/n\Z$}\label{SubsectionAutH}
	As remarked in \cite{DiIs08}, each subgroup $H$ of $\Aut(G)$ determines a corresponding 
	supercharacter theory $(\X_H,\K_H)$ for $G$.  To be more specific,
	$H$ induces a permutation of $\Irr(G)$ while also permuting the conjugacy classes of $G$.
	By Brauer's Lemma \cite[Thm.~6.32, Cor.~6.33]{Isaacs}, the number of $H$-orbits of $\Irr(G)$ equals
	the number of $H$-orbits induced on the set of conjugacy classes of $G$.
	This decomposition yields a supercharacter theory for $G$ where the elements $X_i$ of
	$\X_H = \{X_1,X_2,\ldots,X_r\}$ are $H$-orbits in $\Irr(G)$ and the superclasses $K_i$
	of $\K_H = \{K_1,K_2,\ldots,K_r\}$ are unions of $H$-orbits of conjugacy classes of $G$.
	By construction, each supercharacter \eqref{eq-SCG}
	is constant on each member of $\K_H$.
	
	Fix a positive integer $n$ and let $\tau(n)$ denote the number of divisors of $n$.  Let
	$d_1,d_2,\ldots,d_{\tau(n)}$ denote the divisors of $n$, the exact order being unimportant
	for our purposes at the moment.  Recall that the irreducible characters of $\Z/n\Z$ are precisely the functions
	\begin{equation*}
		\chi_a(x) = e\left(\frac{ax}{n}\right)
	\end{equation*}
	for $a=1,2,\ldots,n$ and that each automorphism of $\Z/n\Z$ is of the form
	\begin{equation*}
		\psi_u(a) = ua
	\end{equation*}
	for some $u$ in $(\Z/n\Z)^{\times}$.  In particular, we note that $\Aut(\Z/n\Z) \cong (\Z/n\Z)^{\times}$.
	
	Next we observe that there exists a $u$ in $(\Z/n\Z)^{\times}$ such that $\psi_u(a) = b$
	if and only if $(a,n)=(b,n)$.  In light of the fact that $\chi_a \circ \psi_u = \chi_{au}$, it is clear that the action of
	$\Aut(\Z/n\Z)$ on $\Irr(\Z/n\Z)$ partitions the irreducible characters
	into a disjoint collection $\X = \{X_1,X_2,\ldots,X_{\tau(n)}\}$ of orbits
	\begin{equation*}
		X_i = \left\{ \chi_a : (a,n) = \frac{n}{d_i} \right\},
	\end{equation*}
	each of which satisfies
	\begin{equation}\label{eq-XDI}
		|X_i| = \phi(d_i),
	\end{equation}
	where $\phi$ denotes the Euler totient function.  Thus
	\begin{equation}\label{eq-SIDF}
		\sigma_i(x) = \sum_{\chi\in X_i} \chi(x) = \sum_{\substack{j=1 \\ (j,n)=\frac{n}{d_i}}}^n e\left( \frac{jx}{n}\right)
		= \sum_{\substack{k=1\\(k,d_i)=1}}^{d_i} e\left( \frac{kx}{d_i} \right) = c_{d_i}(x),
	\end{equation}
	each of which is a Ramanujan sum.
	On the other hand, the action of $\Aut(\Z/n\Z)$ on $\Z/n\Z$ results in a partition $\K = \{K_1,K_2,\ldots,K_{\tau(n)}\}$
	of $\Z/n\Z$ into disjoint orbits
	\begin{equation}\label{eq-KD}
		K_j = \left\{ a \in \Z/n\Z : (a,n)=\frac{n}{d_j} \right\},
	\end{equation}
	each of which satisfies
	\begin{equation}\label{eq-KDI}
		|K_j| = \phi(d_j).
	\end{equation}
	Since each conjugacy class of $\Z/n\Z$ is a singleton, it is clear that the proposed superclass
	\eqref{eq-KD} is the union of conjugacy classes.
	
	Let us pause briefly to note that since $c_n(x)$ is a superclass function with respect to the variable $x$,
	we obtain the following useful fact:
	\begin{equation}\label{eq-DependsGCD}
		\boxed{ c_n\big( (x,n) \big) = c_n(x) .}
	\end{equation}
	In other words, $c_n(x)$ is an \emph{even function modulo $n$} \cite[p.~79]{McCarthyBook}, \cite[p.~15]{SchwarzBook}.
	
	Putting this all together, we obtain the $\tau(n) \times \tau(n)$ supercharacter table $S(n)$
	\begin{equation*}
		\begin{array}{|c||cccc|}
			\hline
			& K_1 & K_2 & \cdots & K_{\tau(n)} \\
			\hline\hline
			\sigma_1 & c_{d_1}( \frac{n}{d_1}) &  c_{d_1}( \frac{n}{d_2}) & \cdots &  c_{d_1}( \frac{n}{d_{\tau(n)}}) \\
			\sigma_2 & c_{d_2}( \frac{n}{d_1}) &  c_{d_2}( \frac{n}{d_2}) & \cdots &  c_{d_2}( \frac{n}{d_{\tau(n)}}) \\
			\vdots & \vdots & \vdots & \ddots & \vdots \\
			\sigma_{\tau(n)} & c_{d_{\tau(n)}}( \frac{n}{d_1}) &  c_{d_{\tau(n)}}( \frac{n}{d_2}) & \cdots &  c_{d_{\tau(n)}}( \frac{n}{d_{\tau(n)}}) \\
			\hline
		\end{array}
	\end{equation*}
	whose $(i,j)$ entry is given by
	\begin{equation}\label{eq-Sij}
		\big[S(n)\big]_{i,j} = c_{d_i}\left( \frac{n}{d_j} \right).
	\end{equation}
	When there is no chance of confusion, we simply write $S=S(n)$.  We leave the
	particular order in which the divisors $d_1,d_2,\ldots,d_{\tau(n)}$ of $n$ are listed unspecified, noting only that
	any pair of orderings lead to two matrices which are similar
	via a suitable permutation matrix.  
	Such relationships between matrices will arise frequently in what follows and we therefore
	introduce the following notation.
	We shall write $A \cong B$ whenever $A$ and $B$ are square matrices such that
	$B = P^{-1}AP$ for some permutation matrix $P$.  Although we use the same symbol
	to denote the isomorphism of groups, the meaning should be clear from context.

	Before moving on, let us recall that pre- and post-multiplying $S$ by the diagonal matrices
	$L$, given by \eqref{eq-LeftMultiplier}, and $R$, given by \eqref{eq-RightMultiplier}, 
	yield the unitary matrix $U = LSR$.  In light of \eqref{eq-XDI} and \eqref{eq-KDI}, it turns out that 
	$L = (\sqrt{n} R)^{-1}$ whence
	\begin{equation}\label{eq-Similar}
		\sqrt{n} U = R^{-1}SR.
	\end{equation}
	In other words, the supercharacter table $S(n)$ is similar to a multiple of the unitary matrix
	$U=U(n)$ given by
	\begin{equation}\label{eq-ExplicitU}
		 \frac{1}{\sqrt{n}}\small
		\begin{bmatrix}
			c_{d_1}( \frac{n}{d_1}) &  c_{d_1}( \frac{n}{d_2}) \sqrt{ \frac{\phi(d_2)}{\phi(d_1)} }& \cdots &  c_{d_1}( \frac{n}{d_{\tau(n)}}) \sqrt{ \frac{\phi(d_{\tau(n)})}{\phi(d_1)} } \\
			c_{d_2}( \frac{n}{d_1})  \sqrt{ \frac{\phi(d_1)}{\phi(d_2)} } &  c_{d_2}( \frac{n}{d_2}) & \cdots &  c_{d_2}( \frac{n}{d_{\tau(n)}})  \sqrt{ \frac{\phi(d_{\tau(n)})}{\phi(d_2)} }\\
			\vdots & \vdots & \ddots & \vdots \\
			c_{d_{\tau(n)}}( \frac{n}{d_1})  \sqrt{ \frac{\phi(d_1)}{\phi(d_{\tau(n)})} }&  c_{d_{\tau(n)}}( \frac{n}{d_2})  \sqrt{ \frac{\phi(d_2)}{\phi(d_{\tau(n)})} }& \cdots &  c_{d_{\tau(n)}}( \frac{n}{d_{\tau(n)}}) \\
		\end{bmatrix}.
	\end{equation}
	
	\begin{Example}
		If $p$ is a prime number, then the two divisors $d_1 = 1$ and $d_2 = p$ of $p$ lead to
		the corresponding superclasses $K_1 = \{p\}$ and $K_2 = \{1,2,\ldots,p-1\}$ of $\Z/ p\Z$.  
		A short computation now reveals that
		\begin{equation}\label{eq-CharacterTablePrime}
			S(p) = \minimatrix{1}{1}{p-1}{-1},\qquad
			U(p) = \frac{1}{\sqrt{p}} \minimatrix{1}{ \sqrt{p-1} }{ \sqrt{p-1} }{-1}.
		\end{equation}
		In particular, observe that $U(p)$ is a selfadjoint unitary involution which has only real entries.  It turns out,
		as we shall see, that this is true for general $U(n)$.
	\end{Example}

\subsection{Orthogonality relations}
	Using the row and column orthogonality relations \eqref{eq-RowOrthogonality} and
	\eqref{eq-ColumnOrthogonality}, we immediately obtain 
	\begin{equation}\label{eq-RRO}
		\boxed{
		\sum_{k|n}   \phi(k) c_{d_i}\left(\frac{n}{k} \right) c_{d_j}\left(\frac{n}{k} \right) = 
		\begin{cases}
			0 & \text{if $i \neq j$},\\
			n\phi(d_i) & \text{if $i =j$},
		\end{cases}
		}
	\end{equation}
	and
	\begin{equation}\label{eq-CCO}
		\boxed{
		\sum_{k|n} \frac{1}{\phi(k)} c_k\left( \frac{n}{d_i} \right)c_k\left( \frac{n}{d_j} \right)=
		\begin{cases}
			0 & \text{if $i \neq j$},\\[2pt]
			\frac{n}{\phi(d_i)} & \text{if $i =j$},
		\end{cases}
		}
	\end{equation}
	respectively.  The first is a well-known identity \cite[Thm.~3.1.e, p.~16]{SchwarzBook} and the second
	is somewhat lesser-known \cite[Ex.~2.22]{McCarthyBook}.  If $d|n$, then letting $d_i = d$ and $d_j = 1$ in \eqref{eq-RRO}
	we obtain \cite[Ex.~2.24]{McCarthyBook}
	\begin{equation*}
		\boxed{
		\sum_{k|n}   \phi(k) c_{d}\left(\frac{n}{k} \right)  = 
		\begin{cases}
			0 & \text{if $d \neq 1$},\\
			n & \text{if $d = 1$}.
		\end{cases}
		}
	\end{equation*}

\section{Multiplicativity and Kronecker products}\label{SectionKronecker}

	In this section, we 
	consider product supercharacter theories and their ramifications for the study of Ramanujan sums.
	In particular, it turns out that many of the peculiar multiplicative properties of Ramanujan sums
	can be easily derived by examining Kronecker products of supercharacter tables.
	In addition to providing simple proofs of many standard identities, our techniques will ultimately permit
	the derivation of many novel identities as well (e.g., the bizarre determinantal formula \eqref{eq-Determinant},
	and the power sum identities of Subsection \ref{SubsectionPower}).

\subsection{Prime powers}\label{SubsectionPrimePowers}
	In the following, we let $p$ denote a fixed prime number.
	For each $\alpha \geq 1$, let us identify the supercharacter table $S=S(p^{\alpha})$
	which arises from the action of $\Aut (\Z / p^{\alpha}\Z) \cong (\Z / p^{\alpha} \Z)^{\times}$ on the group
	$\Z/p^{\alpha}\Z$.    Before proceeding, it is helpful to note that
	\begin{equation}\label{eq-RamanujanPrimePowers}
		c_{p^m}(x) = 
		\begin{cases}
			p^{m-1}(p-1) & \text{if $p^m | x$},\\
			-p^{m-1} & \text{if $p^m \nmid x$ but $p^{m-1} | x$},\\
			0 & \text{otherwise},
		\end{cases}
	\end{equation}
	which can be computed easily from the definition \eqref{eq-RamanujanSumDefinition}
	and the formula for the sum of a finite geometric series.  Among other things, we note that
	\begin{equation}\label{eq-PrimePowerMobiusPhi}
		c_{p^m}(1) = \mu(p^m),\qquad
		c_{p^m}(p^m) = \phi(p^m),
	\end{equation}
	where $\mu(n)$ denotes the \emph{M\"obius $\mu$-function}
	\begin{equation}\label{eq-MobiusDefinition}
		\mu(n) = 
		\begin{cases}
			1 & \text{if $n=1$},\\
			0 & \text{if $n$ is not square-free},\\
			(-1)^{\omega} & \text{if $n$ is the product of $\omega$ distinct primes}.
		\end{cases}
	\end{equation}	

	Since the divisors of $p^{\alpha}$ are precisely the numbers $d_i = p^{i-1}$ for $i=1,2,\ldots,\alpha+1$,
	\eqref{eq-Sij} and \eqref{eq-RamanujanPrimePowers} tell us that the $(i,j)$ entry
	of the $(\alpha+1) \times (\alpha+1)$ matrix $S = S(p^{\alpha})$ is given by
	\begin{equation}\label{eq-PrimePowerSij}
		\big[S(p^{\alpha})\big]_{i,j} = c_{p^{i-1}}( p^{\alpha-j+1}) 
		= 
		\begin{cases}
			1 & \text{if $i =1$},\\
			p^{i-2}(p-1) & \text{if $i+j \leq \alpha+2$},\\
			-p^{i-2} & \text{if $i+j = \alpha+3$},\\
			0 & \text{if $i+j > \alpha+3$}.
		\end{cases}
	\end{equation}
	For instance, the supercharacter tables $S(p^2)$ and $S(p^3)$ are given by
	\begin{equation}\label{eq-AIA}
		\left[
		\begin{array}{c|cc}
			1 & 1 & 1 \\
			p-1 & p-1 & -1 \\
			\hline
			p(p-1) & -p & 0
		\end{array}
		\right],
		\qquad
		\left[
		\begin{array}{c|ccc}
			1 & 1 & 1 & 1\\
			p-1 & p-1 & p-1 & -1 \\
			p(p-1) & p(p-1) & -p & 0\\
			\hline
			p^2(p-1) & -p^2 & 0 & 0
		\end{array}
		\right],
	\end{equation}
	respectively.  In general, 
	$S(p^{\alpha-1})$ appears as the upper-right hand corner of $S(p^{\alpha})$,
	as illustrated in \eqref{eq-AIA}.		
	Let us also note, for future reference, that
	\begin{equation}\label{eq-TraceS}
		\tr S(p^{\alpha}) = 
		\begin{cases}
			p^{\frac{\alpha}{2}} & \text{if $\alpha$ is even},\\
			0 & \text{if $\alpha$ is odd},
		\end{cases}		
	\end{equation}	
	follows from \eqref{eq-PrimePowerSij} and a telescoping series argument.  

	For some purposes, it is more fruitful to consider the associated unitary matrix
	$U = U(p^{\alpha})$, defined by \eqref{eq-ExplicitU}, in place of $S(p^{\alpha})$ itself.  Setting
	$n = p^{\alpha}$, $d_i = p^{i-1}$, and $d_j = p^{j-1}$ in \eqref{eq-ExplicitU}, we find that
	\begin{equation}\label{eq-PPU}
		\big[U(p^{\alpha})\big]_{i,j}
		= \frac{ c_{p^{i-1}} ( p^{\alpha-j+1}) \sqrt{ \phi(p^{j-1})} }{p^{\alpha/2}\sqrt{ \phi(p^{i-1})}} .
	\end{equation}
	A few simple computations reveal that
	\begin{equation}\label{eq-NastyU}
		\big[U(p^{\alpha})\big]_{i,j} = 
		\begin{cases}
			p^{-\frac{\alpha}{2}} & \text{if $i = j = 1$},\\[2pt]
			p^{\frac{j-\alpha-2}{2}}\sqrt{p-1}  & \text{if $i =1$ and $j > 1$},\\[2pt]
			p^{\frac{i-\alpha-2}{2}}\sqrt{p-1} & \text{if $j =1$ and $i > 1$},\\[2pt]
			(p-1)p^{\frac{i+j-\alpha-4}{2}} & \text{if $3\leq i+j \leq \alpha+2$},\\[2pt]
			-\frac{1}{\sqrt{p}} & \text{if $3 \leq i+j = \alpha+3$},\\[2pt]
			0 & \text{if $i+j > \alpha+3$}.
		\end{cases}
	\end{equation}
	Despite its somewhat imposing appearance, the preceding expression tells us that
	the $(\alpha+1) \times (\alpha+1)$ unitary matrix $U$ is selfadjoint and that its lower right $\alpha\times \alpha$ submatrix is a Hankel matrix.
	This is illustrated in the following example.
	
	\begin{Example}
		\begin{equation*}\small
			U(2^6)=
			\left[
			\begin{array}{c|cccccc}
			 \frac{1}{8} & \frac{1}{8} & \frac{1}{4 \sqrt{2}} & \frac{1}{4} & \frac{1}{2 \sqrt{2}} & \frac{1}{2} & \frac{1}{\sqrt{2}} \\[3pt]
			 \hline
			 \frac{1}{8} & \frac{1}{8} & \frac{1}{4 \sqrt{2}} & \frac{1}{4} & \frac{1}{2 \sqrt{2}} & \frac{1}{2} & -\frac{1}{\sqrt{2}} \\[3pt]
			 \frac{1}{4 \sqrt{2}} & \frac{1}{4 \sqrt{2}} & \frac{1}{4} & \frac{1}{2 \sqrt{2}} & \frac{1}{2} & -\frac{1}{\sqrt{2}} & 0 \\[3pt]
			 \frac{1}{4} & \frac{1}{4} & \frac{1}{2 \sqrt{2}} & \frac{1}{2} & -\frac{1}{\sqrt{2}} & 0 & 0 \\[3pt]
			 \frac{1}{2 \sqrt{2}} & \frac{1}{2 \sqrt{2}} & \frac{1}{2} & -\frac{1}{\sqrt{2}} & 0 & 0 & 0 \\[3pt]
			 \frac{1}{2} & \frac{1}{2} & -\frac{1}{\sqrt{2}} & 0 & 0 & 0 & 0 \\[3pt]
			 \frac{1}{\sqrt{2}} & -\frac{1}{\sqrt{2}} & 0 & 0 & 0 & 0 & 0
			\end{array}
			\right].
		\end{equation*}
	\end{Example}

	Since $U = U(n)$ is unitary and selfadjoint, it follows that
	\begin{equation}\label{eq-Involution}
		U^2 = I.
	\end{equation}
	Therefore the only possible eigenvalues of $U$ are $\pm 1$.  The exact multiplicities
	of these eigenvalues can be determined using \eqref{eq-Similar}, which asserts that $p^{\alpha/2} U$ is
	similar to $S$.  It follows from \eqref{eq-TraceS} that
	\begin{equation}\label{eq-TraceU}
		\tr U = 
		\begin{cases}
			1 & \text{if $\alpha$ is even},\\
			0 & \text{if $\alpha$ is odd}.
		\end{cases}				
	\end{equation}
	Since $U$ is $(\alpha+1) \times (\alpha+1)$, it follows that the eigenvalues of $U$
	are $-1$ (multiplicity $\frac{\alpha}{2}$) and $1$ (multiplicity $\frac{\alpha}{2}+1$) if $\alpha$ is even; and
	$\pm 1$ (both with multiplicity $\frac{\alpha+1}{2}$) if $\alpha$ is odd.  Since
	\begin{equation*}
		\lfloor \tfrac{\alpha+1}{2} \rfloor =
		\begin{cases}
			\frac{\alpha}{2} & \text{if $\alpha$ is even},\\[3pt]
			\frac{\alpha+1}{2} & \text{if $\alpha$ is odd},
		\end{cases}
	\end{equation*}
	we conclude that
	\begin{equation}\label{eq-DetU}
		\det U(p^{\alpha}) = (-1)^{ \lfloor \frac{\alpha+1}{2} \rfloor}.
	\end{equation}
	We will make use of this formula later on.

\subsection{Kronecker products}
	Recall that the \emph{Kronecker product}
	$A \otimes B$ of a $m \times n$ matrix $A$ and a $p \times q$ matrix $B$ is the $mp \times nq$ matrix given by
	\begin{equation*}
		A \otimes B = 
		\begin{bmatrix} a_{11} B & \cdots & a_{1n}B \\ \vdots & \ddots & \vdots \\ a_{m1} B & \cdots & a_{mn} B \end{bmatrix}. 
	\end{equation*}
	Whenever the dimensions of the matrices involved are compatible
	we have $A \otimes B \cong B \otimes A$ and 
	\begin{equation*}
		(A\otimes B)(C\otimes D) \cong AC \otimes BD,
	\end{equation*}
	where, as briefly mentioned in Subsection \ref{SubsectionAutH}, $\cong$ denotes similarity
	via a permutation matrix.  Finally, we also recall that 
	if $A$ is $m \times m$ and $B$ is $n \times n$, then 
	\begin{equation}\label{eq-KroneckerDeterminant}
		\det(A \otimes B) = (\det A)^n(\det B)^m
	\end{equation}
	and
	\begin{equation}\label{eq-KroneckerTrace}
		\tr(A \otimes B) = (\tr A)(\tr B).
	\end{equation}

\subsection{Products of supercharacter theories}

	Given two supercharacter theories $(\X_1,\K_1)$ and $(\X_2,\K_2)$ on 
	two finite groups $G_1$ and $G_2$, one can construct a natural product supercharacter 
	theory on $G_1 \times G_2$.  Writing 
	\begin{equation*}
		\K_1 = \{ K_1^{(1)}, K_2^{(1)},\ldots,K_r^{(1)}\},\qquad
		\K_2 = \{ K_1^{(2)}, K_2^{(2)},\ldots,K_s^{(2)}\},
	\end{equation*}
	and
	\begin{equation*}
		\X_1 = \{ X_1^{(1)}, X_2^{(1)},\ldots,X_r^{(1)}\},\qquad
		\X_2 = \{ X_1^{(2)}, X_2^{(2)},\ldots,X_s^{(2)}\},
	\end{equation*}
	we first define
	\begin{equation}\label{eq-KKK}
		\K = \K_1 \times \K_2.
	\end{equation}
	On the other hand, since \cite[Thm.~4.21]{Isaacs} tells us that
	\begin{equation*}
		\Irr(G_1\times G_2) = \Irr (G_1) \times \Irr (G_2),
	\end{equation*}
	it is natural for us to define
	\begin{equation}\label{eq-XXX}
		\X = \X_1 \times \X_2.
	\end{equation}
	A straightforward computation now shows that
	\begin{equation}\label{eq-TensorSigma}
		\sigma_{X_i^{(1)} \times X_j^{(2)}} \big( (g_1,g_2) \big) = \sigma_{X_i^{(1)}}(g_1) \sigma_{X_j^{(2)}}(g_2),
	\end{equation}
	whenever $g_1$ and $g_2$ belong to $G_1$ and $G_2$, respectively.
	In particular, this implies that $\sigma_{X_1 \times X_2}$ is constant on each element of $\K$.
	Since $|\X| = |\X_1||\X_2| = |\K_1||\K_2| = |\K|$, we conclude that 
	$(\X,\K)$ is a supercharacter theory on $G_1 \times G_2$.  
		
	Putting this all together, \eqref{eq-TensorSigma} tells us that if $S_1$ and $S_2$ are the matrices which encode
	the supercharacter tables corresponding to the supercharacter theories
	$(\X_1,\K_1)$ and $(\X_2,\K_2)$ on $G_1$ and $G_2$, respectively, then the Kronecker product $S_1 \otimes S_2$
	encodes the product supercharacter theory $(\X,\K)$ on $G_1 \times G_2$.  To be more specific,
	we list the elements of $\K$ and $\X$ in their respective lexicographic orders induced by the product
	structures \eqref{eq-KKK} and \eqref{eq-XXX}.  In light of \eqref{eq-TensorSigma},
	we see that the resulting supercharacter table $S$ for the product theory $(\X,\K)$ on $G_1\times G_2$ satisfies
	$S \cong S_1 \otimes S_2$.

	The details of the preceding construction were worked out by A.O.F.~Hendrickson, a student of Isaacs,
	in his doctoral thesis \cite[Sect.~2.6]{HendricksonThesis}.  We refer the reader there and to his recent
	paper \cite{Hendrickson} for further information.
	
\subsection{Multiplicativity}
	Recall that if $G_1$ and $G_2$ are finite groups, then 
	\begin{equation*}
		\Aut(G_1) \times \Aut(G_2) \subseteq \Aut(G_1 \times G_2),
	\end{equation*}
	although equality does not hold in general.  
	We are interested here in the special case where $G_1 = \Z / m \Z$, $G_2 = \Z / n \Z$,
	and $(m,n)=1$.  In this setting, 
	\begin{equation}\label{eq-CRT}
		(\Z / m \Z) \times (\Z / n \Z) \cong \Z / mn \Z,
	\end{equation}
	so that if we indulge in a slight abuse of language, we obtain
	\begin{equation*}
		\Aut ( \Z / m \Z) \times \Aut (\Z / n \Z) \subseteq \Aut(\Z / mn \Z),
	\end{equation*}
	or equivalently,
	\begin{equation*}
		(\Z / m\Z)^{\times} \times (\Z/n\Z)^{\times} \subseteq (\Z / mn \Z)^{\times}.
	\end{equation*}
	Since the orders of the preceding groups are $\phi(m)$, $\phi(n)$, and $\phi(mn)$, respectively,
	it follows from the multiplicativity of the Euler totient function that
	\begin{equation*}
		\Aut ( \Z / m \Z) \times \Aut (\Z / n \Z) \cong \Aut(\Z / mn \Z).
	\end{equation*}
	Thus the product supercharacter theory for $\Z /mn \Z$, obtained from the supercharacter theories 
	for $\Z/m\Z$ and $\Z/n\Z$ induced by $\Aut(\Z/m\Z)$ and $\Aut(\Z/n\Z)$, respectively,
	is the same supercharacter theory for $\Z/mn\Z$ which arises
	from the action of $\Aut(\Z /mn\Z$).  In other words, 
	\begin{equation}\label{eq-Smn}
		S(mn)  \cong S(m) \otimes S(n)
	\end{equation}
	whenever $(m,n) = 1$.  In particular, it follows from \eqref{eq-TensorSigma} and the Chinese Remainder Theorem that
	\begin{equation}\label{eq-ccdd}
		\boxed{ c_{mn}(dd') = c_m(d) c_n(d') }
	\end{equation}
	whenever $d$ and $d'$ are positive divisors of $m$ and $n$, respectively.  Indeed,
	first let $G_1 = \Z/m\Z$ and $G_2 = \Z/n\Z$, with
	\begin{equation*}
		K_{\tau(m)}^1 = \{ a \in \Z/m\Z : (a,m) = 1\},\qquad
		K_{\tau(n)}^2 = \{ b \in \Z/n\Z : (a,n) = 1\},
	\end{equation*}
	and observe that the map $\Phi: \Z/m\Z\times \Z/n\Z\to \Z/mn\Z$ defined by
	$\Phi\big( (a,b) \big) = ab\, (\operatorname{mod} mn)$ is an isomorphism.
	In particular, this implies that
	\begin{equation*}
		\Phi\big( K_{\tau(m)}^1 \times K_{\tau(n)}^2\big) = \{ c \in \Z/mn\Z : (c,mn) = 1\}
	\end{equation*}
	and that 
	\begin{equation*}
		\Phi\big( (d,d') \big) = dd'\,(\operatorname{mod} mn)
	\end{equation*}
	whenever $d|m$ and $d'|n$.  Putting this all together and using \eqref{eq-TensorSigma}
	leads to the desired formula \eqref{eq-ccdd}.

	Now recall that when we originally defined the supercharacter table $S(n)$
	for $\Z/n\Z$ (see Subsection \ref{SubsectionAutH}), we were not particular about the manner in which the divisors of $n$ were listed.
	The reason for this lack of specificity is due to the fact that even though the Kronecker product $S(m) \otimes S(n)$
	represents a supercharacter table for $\Z / mn \Z$ arising from the action of $\Aut (\Z / mn\Z)$, 
	the ordering of the superclasses and
	supercharacters in the product table might differ from what one might consider a ``natural'' ordering
	(e.g., the ordering induced by listing the divisors of $mn$ in increasing or decreasing order).  
	However, this poses no difficulty in practice since much of our work will involve similarity invariants of matrices.
	
	\begin{Example}
		For $m = 4$ and $n = 5$ and using the ordered divisor lists $\{1,2,4\}$ and $\{1,5\}$, respectively, we obtain
		\begin{equation*}
			S(4) = 
			\begin{bmatrix}
				 1 & 1 & 1 \\
				 1 & 1 & -1 \\
				 2 & -2 & 0
			\end{bmatrix},\qquad
			S(5) = \minimatrix{1}{1}{4}{-1},
		\end{equation*}
		from \eqref{eq-PrimePowerSij}.  Using the ordered divisor list $\{1,2,4,5,10,20\}$ for $mn=20$ and
		computing $S(20)$ directly from \eqref{eq-Sij} and the definition of Ramanujan sums yields
		\begin{equation*}
			S(20) = \small
			\left[
			\begin{array}{ccc|ccc}
				 1 & 1 & 1 & 1 & 1 & 1 \\
				 1 & 1 & -1 & 1 & 1 & -1 \\
				 2 & -2 & 0 & 2 & -2 & 0 \\
				 \hline
				 4 & 4 & 4 & -1 & -1 & -1 \\
				 4 & 4 & -4 & -1 & -1 & 1 \\
				 8 & -8 & 0 & -2 & 2 & 0
			\end{array}
			\right] = S(5) \otimes S(4).
		\end{equation*}
		In the preceding we have partitioned the matrix $S(20)$ for clarity.
	\end{Example}

	An important consequence of \eqref{eq-ccdd} is the fact that Ramanujan sums $c_n(x)$ are multiplicative 
	with respect to the subscript $n$.
		
	\begin{Theorem}\label{TheoremMultiplicative}
		If $(m,n)=1$, then 
		\begin{equation*}
			\boxed{c_{mn}(x) = c_m(x) c_n(x)}
		\end{equation*}
		for all $x$ in $\Z$.
	\end{Theorem}

	\begin{proof}
		Since $(m,n) = 1$, it follows from \eqref{eq-ccdd} that
		\begin{equation*}
			 c_{mn} \big( (x,mn) \big) = 
			c_{mn}\big( (x,m)(x,n) \big) =
			c_m \big( (x,m) \big) c_n \big( (x,n) \big)
		\end{equation*}
		whence $c_{mn}(x) = c_m(x) c_n(x)$ by \eqref{eq-DependsGCD}.
	\end{proof}
	
	\begin{Corollary}\label{CorollaryMuPhi}
		For $n \geq 1$ we have
		\begin{equation*}
			\boxed{
			c_n(1) = \mu(n),\qquad
			c_n(n) = \phi(n).
			}
		\end{equation*}
		Furthermore, $c_n(x)$ is always an integer.
	\end{Corollary}
	
	\begin{proof}
		The boxed statements follow from \eqref{eq-PrimePowerMobiusPhi} and Theorem \ref{TheoremMultiplicative}.
		The integrality of $c_n(x)$ follows from Theorem \ref{TheoremMultiplicative} and the explicit values given 
		in \eqref{eq-PrimePowerSij}.
	\end{proof}

	The proof of Theorem \ref{TheoremMultiplicative} suggests an interesting variant
	\cite[Ex.~2.2, p.~89]{McCarthyBook}:

	\begin{Theorem}
		If $(mx,ny)=1$, then
		\begin{equation}\label{eq-cmnxy}
			\boxed{ c_{mn}(xy) = c_m(x) c_n(y).}
		\end{equation}
	\end{Theorem}
	
	\begin{proof}
		Since $(m,n)=(x,n) = (y,m) =1$, it follows from \eqref{eq-ccdd} that 
		\begin{equation*}
			c_{mn}\big( (xy,mn) \big) = c_{mn}\big( (x,m)(y,n) \big)
			= c_m\big( (x,m) \big)c_n\big( (y,n) \big),
		\end{equation*}
		whence $c_{mn}(xy) = c_m(x) c_n(y)$ by \eqref{eq-DependsGCD}.
	\end{proof}
	
	\begin{Corollary}\label{CorollarySupermultiplicative}
		If $n = p_1^{\alpha_1} p_2^{\alpha_2}\cdots p_r^{\alpha_r}$ is the canonical factorization
		of $n$ into distinct primes $p_1,p_2,\ldots, p_r$ and $d = p_1^{\beta_1} p_2^{\beta_2} \cdots p_r^{\beta_r}$
		is a divisor of $n$, then
		\begin{equation*}
			c_n(d) = \prod_{\ell=1}^r c_{p_{\ell}^{\alpha_{\ell}}}(p_{\ell}^{\beta_{\ell}}).
		\end{equation*}
	\end{Corollary}

\subsection{Piecing things together}
	The following useful result permits us to deduce a variety of results about Ramanujan
	sums by piecing together our observations from Subsection \ref{SubsectionPrimePowers}.
	In the present setting, recall that product supercharacter theories correspond to
	Kronecker products of supercharacter tables.

	\begin{Theorem}\label{TheoremSU}
		If $n = p_1^{\alpha_1} p_2^{\alpha_2}\cdots p_r^{\alpha_r}$ is the canonical factorization
		of $n$ into distinct primes $p_1,p_2,\ldots, p_r$, then
		\begin{equation}\label{eq-Tensor}
			S(n) \cong \bigotimes_{i=1}^{r} S(p_i^{\alpha_i}),\qquad
			U(n) \cong \bigotimes_{i=1}^{r} U(p_i^{\alpha_i}).		
		\end{equation}
		In particular, $U(n)$ is a selfadjoint, unitary involution whose $(i,j)$ entry is given by
		\begin{equation}\label{eq-Unij}
			\big[U(n)\big]_{i,j} =  \frac{1}{\sqrt{n}}c_{d_i}\left( \frac{n}{d_j}\right) \sqrt{ \frac{\phi(d_j)}{\phi(d_i)} },
		\end{equation}
		where $d_1,d_2,\ldots,d_{\tau(n)}$ are the positive divisors of $n$.
		We also have $S(n)^2 = n I$.
	\end{Theorem}
	
	\begin{proof}
		The first matrix identity in \eqref{eq-Tensor} follows immediately from \eqref{eq-Smn}.
		The second identity in \eqref{eq-Tensor} follows from the first identity, the multiplicativity of
		\eqref{eq-PPU}, and the basic properties
		of the Kronecker product, along with
		\eqref{eq-Similar}, \eqref{eq-RightMultiplier}, and \eqref{eq-KDI}.		
		The fact that $U(n)$ is a selfadjoint involution follows from the fact that each $U(p_i^{\alpha_i})$
		is a selfadjoint involution.  The formula \eqref{eq-Unij} is a simple
		consequence of \eqref{eq-PPU} and the multiplicativity of the Euler totient function.  Finally,
		we note that \eqref{eq-Similar} now implies that $S(n)^2 = nI$.		
	\end{proof}

	As a trivial consequence of Theorem \ref{TheoremSU} we obtain
	\cite[Ex.~2.10]{McCarthyBook}:
	
	\begin{Corollary}
		For $n \geq 1$ we have
		\begin{equation*}
			\boxed{
			\sum_{d|n} c_d\left( \frac{n}{d} \right) = 
			\begin{cases}
			0 & \text{if $n$ is not a perfect square}, \\
			\sqrt{n} & \text{if $n$ is a perfect square}.
			\end{cases}
			}
		\end{equation*}
	\end{Corollary}
	
	\begin{proof}
		Writing $n = p_1^{\alpha_1} p_2^{\alpha_2}\cdots p_r^{\alpha_r}$ and applying
		\eqref{eq-Tensor} we find that
		\begin{equation*}
			\sum_{d|n} c_d\left( \frac{n}{d} \right) 
			= \tr S(n) 
			= \tr\Big( \bigotimes_{i=1}^{r} S(p_i^{\alpha_i}) \Big)
			= \prod_{i=1}^{r} \tr S(p_i^{\alpha_i}).
		\end{equation*}
		The result now follows immediately from \eqref{eq-TraceS}.
	\end{proof}

	The following corollary of Theorem \ref{TheoremSU} appears to be novel, as we
	were unable to find it in our extensive search of the literature.  In particular, although the
	magnitude of the following determinant is possible to conjecture based on numerical evidence,
	the sign of the determinant is determined by a rather complicated formula which seems difficult to 
	arrive at using other means.

	\begin{Corollary}
		If $n = p_1^{\alpha_1} p_2^{\alpha_2}\cdots p_r^{\alpha_r}$ is the canonical factorization
		of $n$ into distinct primes $p_1,p_2,\ldots, p_r$, then
		\begin{equation}\label{eq-Determinant}
			\boxed{
			\det \Big[ c_{d_i}\Big(\frac{n}{d_j}\Big) \Big]_{i,j=1}^{\tau(n)}
			= n^{\frac{\tau(n)}{2}} (-1)^{\textstyle\sum_{i=1}^r\lfloor \frac{\alpha_i+1}{2} \rfloor  \frac{\tau(n)}{\alpha_i+1}}
			}
		\end{equation}
		where $\tau(n)$ denotes the number of positive divisors 
		$d_1,d_2,\ldots,d_{\tau(n)}$ of $n$.
	\end{Corollary}
	
	\begin{proof}
		Simply observe that
		\begin{align*}
			\det \Big[ c_{d_i}\Big(\frac{n}{d_j}\Big) \Big]_{i,j=1}^{\tau(n)}
			&= \det S(n) && \text{by \eqref{eq-Sij}}\\
			&= n^{\frac{\tau(n)}{2}} \det U(n)  &&\text{by \eqref{eq-Similar}}\\
			&= n^{\frac{\tau(n)}{2}} \det\Big( \bigotimes_{i=1}^{r} U(p_i^{\alpha_i}) \Big) &&\text{by \eqref{eq-Tensor}} \\
			&= n^{\frac{\tau(n)}{2}} \prod_{i=1}^{r}\Big( \det U(p_i^{\alpha_i}) \Big)^{ \frac{\tau(n)}{\alpha_i+1}} 
				&& \text{by \eqref{eq-KroneckerDeterminant}}\\
			&= n^{\frac{\tau(n)}{2}} \prod_{i=1}^{r}\Big( (-1)^{ \lfloor \frac{\alpha_i+1}{2} \rfloor} \Big)^{ \textstyle\frac{\tau(n)}{\alpha_i+1}} 
				&& \text{by \eqref{eq-DetU}}\\
			&= n^{\frac{\tau(n)}{2}} (-1)^{\textstyle\sum_{i=1}^r\lfloor \frac{\alpha_i+1}{2} \rfloor  \frac{\tau(n)}{\alpha_i+1}}.\qedhere
		\end{align*}
	\end{proof}

\subsection{Reciprocity and von Sterneck's formula}
	Our next result requires no proof, for it follows immediately from \eqref{eq-ExplicitU}
	and the fact that $U = U^T$.
	
	\begin{Theorem}[Reciprocity Formula]
		If $d$ and $d'$ are positive divisors of $n$, then
		\begin{equation}\label{eq-Reciprocity}
			\boxed{ c_{d}\left( \frac{n}{d'} \right) \phi(d') = c_{d'} \left( \frac{n}{d} \right) \phi(d). }
		\end{equation}
	\end{Theorem}
	
	Although we have been unable to find \eqref{eq-Reciprocity} in the literature, given the long
	and storied history of the Ramanujan sum, it is certainly possible that we are not the first to have discovered it.  
	In fact, various other reciprocity formulas have also been discussed \cite{JohnsonReciprocity, JohnsonReciprocity2}.
	For our purposes, the importance of \eqref{eq-Reciprocity}
	lies in the fact that it provides a one-line proof of the following important formula
	\cite[Thm.~272]{Hardy}, \cite[Cor.~2.4]{McCarthyBook}, \cite[p.~40]{SchwarzBook}.
	
	\begin{Corollary}[von Sterneck's Formula]
		For $n,x \in \Z$ with $n \geq 1$ we have
		\begin{equation}\label{eq-vonSterneck}
			\boxed{ c_n(x) = \frac{ \mu\left( \frac{n}{(n,x)}\right) \phi(n) }{ \phi\left( \frac{n}{(n,x)}\right)}.}
		\end{equation}
	\end{Corollary}

	\begin{proof}
		Let $d' = n$ and $d = n/(n,x)$ in \eqref{eq-Reciprocity} and use \eqref{eq-DependsGCD}.
	\end{proof}	
	
	Let us make a few historical remarks concerning the preceding formula.
	The peculiar arithmetic function on the right-hand side of \eqref{eq-vonSterneck} is sometimes 
	called \emph{von Sterneck's function}.  It was first studied by R.D.~von Sterneck (1902) \cite{vonSterneck},
	independently of Ramanujan sums, which first rose to prominence with Ramanujan's seminal
	paper (1918) \cite[Paper 21]{Ramanujan}.  It has frequently been claimed that the fact that von Sterneck's function equals
	$c_n(x)$ was first observed by O.~H\"older (1936) \cite{Holder}.  	
	However, Peter van der Kamp was kind enough to inform us that 
	Kluyver  (1906) \cite[p.~410]{Kluyver} had already discovered the equality \eqref{eq-vonSterneck} 
	some thirty years before H\"older's paper appeared.

	Before proceeding, let us also remark that Corollary \ref{CorollaryMuPhi} follows immediately
	from von Sterneck's formula by setting $x=1$ and $x=n$, respectively, in \eqref{eq-vonSterneck}.
	
	\begin{Corollary}
		If $n \geq 1$, then
		\begin{equation*}
			\boxed{
			\sum_{d|n,d'|n}   c_d\left( \frac{n}{d'} \right)c_{d'}\left( \frac{n}{d} \right)  =n \tau(n).
			}
		\end{equation*}
	\end{Corollary}
	
	\begin{proof}
		Letting $I$ denote the $\tau(n) \times \tau(n)$ identity matrix, it follows from \eqref{eq-Unij} that
		\begin{equation*}
			\tau(n) = \tr I = \tr U^*U =  \sum_{i,j=1}^{\tau(n)} \left(\big[ U(n)\big]_{i,j}\right)^2
			= \sum_{d|n,d'|n} \frac{1}{n} \left(c_d\left( \frac{n}{d'} \right) \right)^2 \frac{ \phi(d')}{\phi(d)}.
		\end{equation*}
		The desired formula now follows from \eqref{eq-Reciprocity}.
	\end{proof}

\subsection{A mixed orthogonality relation}
	An immediate consequence of Theorem \ref{TheoremSU} is the following
	orthogonality relation in which the parameters $d$ and $d'$ play quite different roles.  
	The text \cite[Thm.~2.8]{McCarthyBook} devotes almost two pages to its proof.

	\begin{Theorem}[Mixed Orthogonality Relation]\label{TheoremMixed}
		If $n \geq 1$, $d|n$, and $d'|n$, then
		\begin{equation}\label{eq-Mixed}
			\boxed{
			\frac{1}{n} \sum_{k|n} c_{d}\left( \frac{n}{k} \right) c_{k} \left( \frac{n}{d'} \right)
			= \delta_{d,d'}. }
		\end{equation}
	\end{Theorem}

	\begin{proof}
		Compute the $(i,j)$ entry in the equation $S(n)^2 = nI$.
	\end{proof}

	\begin{Corollary}\label{CorollaryColumnSum}
		If $x$ is an integer, then
		\begin{equation}\label{eq-UIK}
			\boxed{
			\sum_{k|n} c_k(x) = 
			\begin{cases}
				n &\text{if $n|x$},\\
				0 & \text{if $n\nmid x$}.
			\end{cases}
			}
		\end{equation}
	\end{Corollary}

	\begin{proof}
		Set $d = 1$ and $d' = n/(x,n)$ in \eqref{eq-Mixed} to obtain
		$\sum_{k|n} c_k\big( (x,n) \big) = n \delta_{ \frac{n}{(x,n)},1}$.
		Now apply \eqref{eq-DependsGCD} to obtain \eqref{eq-UIK}.
	\end{proof}
	

	\begin{Corollary}
		If $n \geq 1$ and $d|n$, then
		\begin{equation*}
			\boxed{
			\sum_{k|n} c_d\left( \frac{n}{k} \right) \mu(k) = 
			\begin{cases}
				n & \text{if $d = n$},\\
				0 & \text{if $d \neq n$}.
			\end{cases}
			}
		\end{equation*}
	\end{Corollary}
	
	\begin{proof}
		Let $d'=n$ in \eqref{eq-Mixed} and use Corollary \ref{CorollaryMuPhi}.
	\end{proof}

	One says that a function $f:\Z\to\C$ is \emph{even modulo $n$} if
	\begin{equation*}
		f\big( (n,x) \big) = f(x)
	\end{equation*}
	holds for all integers $x$ \cite[p.~79]{McCarthyBook}, \cite[p.~15]{SchwarzBook}.
	As we noted in \eqref{eq-DependsGCD}, since $c_n(x)$ is a superclass function with respect to the variable $x$,
	it follows that $c_n(x)$ is even modulo $n$.  We remark that the following theorem 
	\cite[Thm.~2.9]{McCarthyBook} has a trivial proof based upon supercharacter theory.  
	In contrast, the standard proof requires several pages of straightforward but 
	tedious manipulations.

	\begin{Theorem}\label{TheoremEven}
		If $f:\Z\to\C$ is an even function modulo $n$, then $f$ can be written uniquely in the form
		\begin{equation}\label{eq-Expansion}
			f(x) = \sum_{d|n} \alpha(d) c_d(x)
		\end{equation}
		where the coefficients $\alpha(d)$ are given by
		\begin{equation*}	
			\alpha(d) = \frac{1}{n} \sum_{k|n} f\left( \frac{n}{k}\right) c_{k}\left( \frac{n}{d} \right).
		\end{equation*}
	\end{Theorem}
	
	\begin{proof}
		Since the Ramanujan sums $c_{d_i}(x) = \sigma_i(x)$ form a basis for the space of
		all superclass functions by \cite[Thm.~2.2]{DiIs08}, a unique expansion of the form \eqref{eq-Expansion}
		exists.  The formula for the coefficients follows immediately from \eqref{eq-Mixed}.
	\end{proof}
	

	As Hardy observes in the notes to \cite[Paper 21]{HardyRamanujan}, the following important result 
	was first obtained by Kluyver (1906) \cite[p.~410]{Kluyver}.  It also appears in the more recent texts
	\cite[Prop.~2.1]{McCarthyBook}, \cite[Thm.~A.24]{Nathanson}, and \cite[Thm.~3.1b]{SchwarzBook}.
	
	\begin{Theorem}[Kluyver]\label{TheoremKluyver}
		If $n,x \in \Z$ and $n \geq 1$, then
		\begin{equation}\label{eq-Kluyver}
			\boxed{ c_n(x) = \sum_{ d | (n,x) } \mu \left( \frac{n}{d} \right)d.}
		\end{equation}
	\end{Theorem}

	\begin{proof}
		Applying the M\"obius inversion formula to \eqref{eq-UIK} it follows that
		\begin{equation*}
			c_n(x)
			= \sum_{d|n} \mu\left( \frac{n}{d} \right) \sum_{k|d} c_k(x) 
			= \sum_{d|n, d|x} \mu\left( \frac{n}{d} \right) d 
			=  \sum_{ d | (n,x) } \mu \left( \frac{n}{d} \right)d. \qedhere
		\end{equation*}
	\end{proof}
	
	\begin{Corollary}
		For all $m,n \geq 1$, the inequality 
		\begin{equation*}
			\boxed{ |c_n(m)| \leq \sigma(m)}
		\end{equation*}
		holds.  Here $\sigma(m)$ denotes the sum of the divisors of $m$.
	\end{Corollary}

\section{Superclass arithmetic}\label{SectionSuperclass}

	In this section, we explore several further properties of Ramanujan sums which 
	can be deduced by studying \emph{superclass arithmetic} on $\Z/n\Z$.  
	Although this approach has been attempted sporadically throughout the years 
	\cite{Haukkanen,Kesava, Kesava2,Maze,Ramanathan2}, 
	these earlier authors did not have the benefit of the general theory of supercharacters.

	We state the following preparatory lemmas in full generality, noting that they apply to any finite group 
	$G$ (we require only the case $G = \Z/n\Z$).  In fact, a more primitive approach was recently undertaken in \cite{CKS}
	where Kloosterman sums were considered in the context of classical character theory.

\subsection{Superclass constants and simultaneous diagonalization}
	Suppose that $G$ is a finite group with supercharacter theory $(\X_H,\K_H)$ generated
	by the action of some subgroup $H$ of $\Aut(G)$ (see Subsection \ref{SubsectionAutH}).
	In particular, we obtain from the action of $H$ on $G$ a partition
	$\X = \{X_1,X_2,\ldots,X_r\}$ of $\Irr(G)$ with corresponding supercharacters 
	\eqref{eq-SCG} and a compatible partition $\K = \{K_1,K_2,\ldots,K_r\}$ of $G$ into superclasses.
	We also note that the superclass sums
	\begin{equation*}
		\hat{K}_i = \sum_{g \in K_i} g
	\end{equation*}
	belong to the center ${\bf Z}(\C[G])$ of the group algebra $\C[G]$.  Indeed, each
	$K_i$ is a union of conjugacy classes of $G$ and it is well-known that the
	corresponding class sums each belong to ${\bf Z}(\C[G])$ \cite[Thm.~2.4]{Isaacs}.
	
	Although the following lemma is a special case of \cite[Cor.~2.3]{DiIs08},
	we provide a brief explanation for the sake of completeness.  Since we are, for the moment, working
	under the assumption that $G$ is an arbitrary finite group, we write the group operation
	\emph{multiplicatively}.  When we later apply the following two results to $\Z/n\Z$, the group operation
	will be \emph{addition} modulo $n$.
	
	\begin{Lemma}\label{LemmaPreliminary}
		Fix some element $z$ in $K_k$ and let $a_{i,j,k}$ denote the number of solutions 
		$(x_i,y_j) \in K_i \times K_j$ to the equation $xy = z$.  The \emph{superclass constants} $a_{i,j,k}$ satisfy
		\begin{equation}\label{eq-cijk}
			\hat{K}_i \hat{K}_j = \sum_{k=1}^r a_{i,j,k} \hat{K}_k
		\end{equation}
		for $1 \leq i,j,k \leq r$.  
	\end{Lemma}

	\begin{proof}
		It suffices to prove that $a_{i,j,k}$ does not depend upon the particular representative $z$
		of $K_k$ which is chosen.	  Suppose that $z_1$ and $z_2$ belong to $K_k$.  By definition of the supercharacter
		theory $(\X_{H}, \K_{H})$, there exists an automorphism $\varphi:G\to G$ belonging to $H$
		and an element $g$ of $G$ such that
		\begin{equation*}
			\varphi(z_1) = \varphi(g)^{-1}z_2\varphi(g).
		\end{equation*}
		If $(x_1,y_1) \in K_i \times K_j$ is a solution to the equation $xy =z_1$, then
		\begin{equation*}
			\varphi(x_1)\varphi(y_1) =\varphi(z_1)= \varphi(g)^{-1} z_2 \varphi(g),
		\end{equation*}
		from which it follows that the elements $x_2 = \varphi(gx_1g^{-1})$ of $K_i$
		and $y_2 = \varphi(gy_1g^{-1})$ of $K_j$ satisfy $x_2 y_2 = z_2$.  Thus there is a
		bijection between solutions $(x_1,y_1) \in K_i \times K_j$ of $xy=z_1$ and
		solutions $(x_2,y_2) \in K_i \times K_j$ of $xy = z_2$.
	\end{proof}

	The following theorem is partly inspired by the corresponding result from classical character theory 
	\cite[Section 33]{CR62}, \cite[Lem.~3.1]{CKS}, \cite[Lem.~4]{Ku75}.
	Since we require a supercharacter version of this result, 
	we provide a detailed proof.  As with the preceding lemma, we work with an arbitrary finite group $G$,
	maintaining the notation and conventions established at the beginning of this subsection.

	\begin{Theorem}\label{TheoremPowerful}
		Let $M_i = (a_{i,j,k})_{j,k=1}^r$.
		If $W = (w_{j,k})_{j,k=1}^r$ denotes the $r \times r$ matrix with entries
		\begin{equation}\label{eq-OmegaDefinition}
			w_{j,k} = \frac{|K_j|\sigma_k(K_j)}{ \sum_{\chi \in X_k} \chi(1)}, 
		\end{equation}
		and $D_i = \operatorname{diag}(w_{i,1}, w_{i,2}, \ldots, w_{i,r})$, then $W$ is invertible and
		\begin{equation}\label{eq-Similarity}
			M_i W = W D_i  
		\end{equation}
		for $i=1,2,\ldots,r$.  In particular, the matrices $M_1,M_2,\ldots,M_r$ are simultaneously diagonalizable
		and commute with each other.
	\end{Theorem}

	\begin{proof}
		 Applying $\sigma_k$ to the superclass sum $\hat{K}_j$ we first note that
		\begin{equation}\label{eq-JFC05}
			\sigma_k(\hat{K}_j)  = |K_j| \sigma_k(K_j)
		\end{equation}
		since $\sigma_k$ is a superclass function which assumes the constant value $\sigma_k(K_j)$ on
		the superclass $K_j$.  Next let $\pi_k$ be the matrix representation of $G$ given by
		\begin{equation}\label{eq-JFC03}
			\pi_k = \bigoplus_{\chi \in X_k} \chi(1) \pi_{\chi}
		\end{equation}
		where $\pi_{\chi}$ is an irreducible matrix representation whose character $\chi$ belongs to $X_k$.
		Since each $\hat{K}_j$ belongs to ${\bf Z}(\C[G])$ 
		and each $\pi_{\chi}$ is irreducible, there exist constants $w_{j,k}^{\chi}$ such that
		\begin{equation}\label{eq-JFC04}
			\pi_{\chi}(\hat{K}_j) = \frac{w_{j,k}^{\chi}}{\chi(1)} I_{\chi(1)},
		\end{equation}
		where $I_{\chi(1)}$ denotes the $\chi(1) \times\chi(1)$ identity matrix.
		Taking the trace of both sides of  the preceding equation we find that
		\begin{equation}\label{eq-JFC02}
			\chi(\hat{K}_j) = w_{j,k}^{\chi}.
		\end{equation}
	
		We now claim that $w_{j,k}^{\chi}$ is independent of which particular irreducible character
		$\chi$ in $X_k$ is chosen.  Indeed, if $\chi$ and $\chi'$
		belong to $X_k$, then there exists an automorphism $\varphi:G\to G$ belonging to
		$H$ such that $\chi = \chi' \circ \varphi$.  Therefore
		\begin{equation*}
			\chi(\hat{K}_j)= \sum_{g \in K_j} \chi(g) = \sum_{g\in K_j} \chi'(\varphi(g)) = \sum_{h \in \varphi(K_j)} \chi'(h) 
			=\sum_{h \in K_j} \chi'(h) = \chi'(\hat{K}_j)
		\end{equation*}
		since $\varphi$ permutes the conjugacy classes which constitute the superclass 
		$K_j$.  In the following, we now write $w_{j,k}$ in place of $w_{j,k}^{\chi}$.

		Substituting \eqref{eq-JFC04} into \eqref{eq-JFC03} we find that
		\begin{equation}\label{eq-JFC01}
			\pi_k(\hat{K}_j) =\bigoplus_{\chi \in X_k}  w_{j,k}  I_{\chi(1)}.
		\end{equation}
		Taking the trace of the preceding yields
		\begin{equation*}
			\tr \pi_k(\hat{K}_j) = \sum_{\chi \in X_k} \chi(1) w_{j,k} = \sum_{\chi \in X_k} \chi(1)\chi(\hat{K}_j)
			=\sigma_k(\hat{K}_j)
		\end{equation*}
		by the definition \eqref{eq-SCG} of the supercharacter $\sigma_k$.
		Returning to \eqref{eq-JFC05} and using the preceding we find that
		\begin{equation*}
			|K_j| \sigma_k(K_j) = \sigma_k(\hat{K}_j) = w_{j,k} \sum_{\chi \in X_k} \chi(1), 
		\end{equation*}
		from which we obtain the desired formula \eqref{eq-OmegaDefinition} for $w_{j,k}$.
		
		We now need to verify that the simultaneous diagonalization \eqref{eq-Similarity} holds.
		Applying $\pi_{\ell}$ to \eqref{eq-cijk} and using \eqref{eq-JFC01} we see that
		\begin{equation*}
			 \left(\bigoplus_{\chi \in X_{\ell}} w_{i,\ell}I_{\chi(1)}\right) \left(  \bigoplus_{\chi \in X_{\ell}} w_{j,\ell}I_{\chi(1)}\right)
			 = \sum_{k=1}^r a_{i,j,k} \left( \bigoplus_{\chi \in X_{\ell}} w_{k,\ell}I_{\chi(1)} \right).
		\end{equation*}
		Considering the direct summand corresponding to an arbitrary $\chi$ in $X_{\ell}$ we see that
		\begin{equation*} 
			w_{i,\ell}  I_{\chi(1)} w_{j,\ell} I_{\chi(1)} =\sum_{k=1}^r a_{i,j,k} w_{k,\ell}I_{\chi(1)},
		\end{equation*}
		which in turn implies that
		\begin{equation*}
			\sum_{k=1}^r a_{i,j,k} w_{k,\ell} =  w_{j,\ell} w_{i,\ell}.
		\end{equation*}
		To conclude the proof, we observe that the preceding equation is simply 
		the $(j,\ell)$ entry of the matrix equation \eqref{eq-Similarity}.
	\end{proof}

\subsection{The matrices $M_i(p^{\alpha})$}\label{SubsectionMPA}
	We are now ready to discuss the matrices $M_i = M_i(n)$, as defined in Theorem \ref{TheoremPowerful},
	which arise from the supercharacter theory for $G=\Z/n\Z$ induced by $H = \Aut(G)$
	(described in Subsection \ref{SubsectionAutH}).  We do this first for prime powers $n = p^{\alpha}$.
	
	Let us first note that the divisors of $p^{\alpha}$ are the $\alpha+1$ numbers
	$d_i = p^{i-1}$ for $i = 1,2,\ldots,\alpha+1$.  In light of \eqref{eq-KD}, this yields the corresponding superclasses
	\begin{equation}\label{eq-KiDFN}
		K_i = \{ a \in \Z/p^{\alpha}\Z : (a,p^{\alpha}) = p^{{\alpha}-i+1} \} = \{ xp^{\alpha-i+1} \in \Z/p^{\alpha}\Z : p \nmid x\},
	\end{equation}
	each of which satisfies 
	\begin{equation*}
		|K_i| = \phi(p^{i-1})
	\end{equation*}
	by \eqref{eq-KDI}.
	Fixing some $z$ in $K_k$, we let $a_{i,j,k}$ denote the number of solutions
	$(x,y)$ in $K_i \times K_j$ to the equation 
	\begin{equation}\label{eq-xyz}
		x+y = z.
	\end{equation}
	Recall that since Lemma \ref{LemmaPreliminary} concerns general groups, the corresponding equation $xy=z$
	was written in multiplicative notation.  However, since we are considering only abelian groups,
	we choose now to employ additive notation.

	We claim that the matrix $M_i(p^{\alpha}) = [a_{i,j,k}]_{j,k=1}^{\alpha+1}$ is given by
	\begin{equation}\label{eq-MPN}
		\footnotesize
		\left[
		\begin{array}{cccc|c|cccc}
			\0 & \0 & \gcdots & \0 & 1 & \0 & \0 & \gcdots & \0\\
			\0 & \0 & \gcdots & \0 & \phi(p) & \0 & \0 & \gcdots & \0 \\
			\gvdots & \gvdots & \gddots & \gvdots & \vdots & \gvdots & \gvdots & \gddots & \gvdots \\
			\0 & \0 & \gcdots & \0 & \phi(p^{i-2}) & \0 & \0 & \gcdots & \0 \\
			\hline
			\!\!\!\phi(p^{i-1}) & \!\!\!\phi(p^{i-1}) & \cdots & \!\!\!\phi(p^{i-1}) & p^{i-1} - 2p^{i-2} & \0 & \0 & \gcdots & \0 \\
			\hline
			\0 & \0 & \gcdots & \0 & \0 & \phi(p^{i-1}) & \0 & \gcdots & \0 \\
			\0 & \0 & \gcdots & \0 & \0 &\0 & \!\!\! \phi(p^{i-1}) & \gcdots & \0 \\
			\gvdots & \gvdots & \gddots & \gvdots & \gvdots & \gvdots & \gvdots & \ddots & \gvdots \\
			\0 & \0 & \gcdots & \0 & \0 & \0 & \0 & \gcdots &\!\!\! \phi(p^{i-1})\\
		\end{array}
		\right],
	\end{equation}
	where the $i$th row and column of $M_i$ are singled out.  Although the computations involved are elementary,
	we feel compelled to provide a complete justification of \eqref{eq-MPN} since so many of our upcoming results
	depend upon this formula.  The reader is invited to consult Appendix \ref{SectionPrimePowerAlpha}
	for the details.

	The matrix $W = W(p^{\alpha})$ described by Theorem \ref{TheoremPowerful} is somewhat easier to describe.
	In fact, we claim that 
	\begin{equation}\label{eq-WPASPA}
		W(p^{\alpha}) = S(p^{\alpha}),
	\end{equation}
	the supercharacter table for $\Z/p^{\alpha}\Z$ discussed 
	in Subsection \ref{SubsectionPrimePowers}.  To see this, simply note that
	\begin{align*}\qquad\qquad\qquad
		[W(p^{\alpha})]_{j,k}
		&= \frac{ |K_j| \sigma_k(K_j)}{\sum_{\chi \in X_k} \chi(1)} && \text{by \eqref{eq-OmegaDefinition}}\\
		&= \frac{ \phi(d_j) c_{d_k}( \frac{n}{d_j} ) }{\phi(d_k)} && \text{by  \eqref{eq-XDI}, \eqref{eq-SIDF}, \eqref{eq-KDI}}\\
		&= c_{d_j}\Big(\frac{n}{d_k}\Big). && \text{by \eqref{eq-Reciprocity}}
	\end{align*}
	
	Finally, there are the diagonal matrices $D_i = D_i(p^{\alpha})$.  By Theorem \ref{TheoremPowerful} we have
	\begin{equation}\label{eq-PrimePowerD}
		[D_i(p^{\alpha})]_{j,k} 
		= \delta_{j,k} c_{d_i}\left( \frac{p^{\alpha}}{d_k}\right)
		= \delta_{j,k} c_{p^{i-1}}(p^{\alpha-k+1}).
	\end{equation}

	\begin{Example}\label{ExamplePPM1}
		The case $p = 3$ and $\alpha = 4$ yields the divisors $d_1=1$, $d_2 =3$,
		$d_3 = 9$, $d_4 = 27$, and $d_5 = 81$ of $p^{\alpha} = 81$.  The corresponding matrices
		$M_1,M_2,M_3,M_4,M_5$ 
		are displayed below.
		\begin{equation*}\small
			\underbrace{
			\left[
			\begin{array}{c|cccc}
				 1 & \0 & \0 & \0 & \0 \\
				 \hline
				 \0 & 1 & \0 & \0 & \0 \\
				 \0 & \0 & 1 & \0 & \0 \\
				 \0 & \0 & \0 & 1 & \0 \\
				 \0 & \0 & \0 & \0 & 1
			\end{array}
			\right]}_{M_1}\qquad
			\underbrace{ 
			\left[
			\begin{array}{c|c|ccc}
				 \0 & 1 & \0 & \0 & \0 \\
				 \hline
				 2 & 1 & \0 & \0 & \0 \\
				 \hline
				 \0 & \0 & 2 & \0 & \0 \\
				 \0 & \0 & \0 & 2 & \0 \\
				 \0 & \0 & \0 & \0 & 2
			\end{array}
			\right]}_{M_2}\qquad
			\underbrace{
			\left[
			\begin{array}{cc|c|cc}
				 \0 & \0 & 1 & \0 & \0 \\
				 \0 & \0 & 2 & \0 & \0 \\
				 \hline
				 6 & 6 & 3 & \0 & \0 \\
				 \hline
				 \0 & \0 & \0 & 6 & \0 \\
				 \0 & \0 & \0 & \0 & 6
			\end{array}
			\right]
			}_{M_3}
		\end{equation*}
		\begin{equation*}\small
			\underbrace{
			\left[
			\begin{array}{ccc|c|c}
				 \0 & \0 & \0 & 1 & \0 \\
				 \0 & \0 & \0 & 2 & \0 \\
				 \0 & \0 & \0 & 6 & \0 \\
				 \hline
				 18 & 18 & 18 & 9 & \0 \\
				 \hline
				 \0 & \0 & \0 & \0 & 18
			\end{array}
			\right]}_{M_4}\qquad
			\underbrace{
			\left[
			\begin{array}{cccc|c}
				 \0 & \0 & \0 & \0 & 1 \\
				 \0 & \0 & \0 & \0 & 2 \\
				 \0 & \0 & \0 & \0 & 6 \\
				 \0 & \0 & \0 & \0 & 18 \\
				 \hline
				 54 & 54 & 54 & 54 & 27
			\end{array}
			\right]}_{M_5}.
		\end{equation*}
		These matrices each satisfy $M_i W = W D_i$ where
		\begin{equation*}
			W =\small
			\begin{bmatrix}
				 1 & 1 & 1 & 1 & 1 \\
				 2 & 2 & 2 & 2 & -1 \\
				 6 & 6 & 6 & -3 & 0 \\
				 18 & 18 & -9 & 0 & 0 \\
				 54 & -27 & 0 & 0 & 0
			\end{bmatrix}
		\end{equation*}
		is the supercharacter table $S(81)$ and
		\begin{align*}
			D_1 &= \diag(1,1,1,1,1) , \\
			D_2 &= \diag(2,2,2,2,-1), \\
			D_3 &= \diag(6,6,6,-3,0),\\
			D_4 &= \diag(18,18,-9,0,0),\\
			D_5 &= \diag(54,27,0,0,0).
		\end{align*}
	\end{Example}
	
	Before proceeding, let us make a few remarks about the matrices \eqref{eq-MPN}.  First observe that
	$M_i(p^{\alpha})$ is a multiple of a stochastic matrix since each column of $M_i(p^{\alpha})$ sums to $\phi(p^{i-1})$.  
	Indeed, we only need verify this for the $i$th column of $M_i(p^{\alpha})$:
	\begin{align*}
		&1 + \phi(p) + \phi(p^2) + \cdots + \phi(p^{i-2}) + (p^{i-1} - 2p^{i-2})\\
		&\qquad= 1 + (p-1) + (p^2-p)+\cdots + (p^{i-2}-p^{i-3}) + (p^{i-1} - 2p^{i-2})\\
		&\qquad= p^{i-1} - p^{i-2} \\
		&\qquad= \phi(p^{i-1}).
	\end{align*}
	Let us also note that
	\begin{equation}\label{eq-StochasticSum}
		\sum_{i=1}^{\alpha+1} M_i(p^{\alpha}) =\small
		\begin{bmatrix}
			1 & 1 & \cdots & 1 \\
			\phi(p) & \phi(p) & \cdots & \phi(p) \\
			\vdots & \vdots & \ddots & \vdots \\
			\phi(p^{\alpha}) & \phi(p^{\alpha}) & \cdots & \phi(p^{\alpha}) \\
		\end{bmatrix}.
	\end{equation}
	This can be deduced from \eqref{eq-MPN} and the fact that for $j \neq k$, the corresponding off-diagonal matrix entry
	$[M_i(p^{\alpha})]_{j,k}$ is nonzero for only a single value of $i$.  Finally, we observe that
	\begin{equation}\label{eq-Commute}
		1\leq i<j \leq \alpha+1\quad\implies\quad M_i M_j  = \phi(p^{i-1})M_j = M_j M_i
	\end{equation}
	follows from a straightforward computation.

\subsection{The matrices $M_i(n)$ for general $n$}\label{SubsectionMPN}
	Having computed the matrices $M_i(p^{\alpha})$ for prime powers, we now turn to the problem of computing
	$M_i(n)$ for general $n$ and harnessing the power of Theorem \ref{TheoremPowerful} to produce
	new identities for Ramanujan sums.
	 The approach is straightforward enough, for the simultaneous diagonalization \eqref{eq-Similarity} 
	``tensors'' in the expected manner, much as the supercharacter tables did in Theorem \ref{TheoremSU}.  
	The difficulty in establishing this is mostly notational.  We therefore take a moment to
	introduce the somewhat elaborate notation which is required.
	
	Suppose that $(m,n)=1$ and let $d_1,d_2,\ldots,d_{\tau(m)}$ and $e_1,e_2,\ldots,e_{\tau(n)}$ 
	be ordered lists of the divisors of $m$ and $n$, respectively.  Having specified an order to the divisors of 
	$m$ and $n$, we can generate the matrices $M_i(m), W(m),D_i(m)$ and $M_{i'}(n),W(n), D_{i'}(n)$, respectively,
	as defined by Theorem \ref{TheoremPowerful} (i.e., apply Theorem \ref{TheoremPowerful} separately to the two groups
	$\Z/m\Z$ and $\Z/n\Z$, each endowed with the supercharacter theory described in Subsection \ref{SubsectionAutH}).
	
	Next we observe that the divisors of $mn$ are clearly the
	$\tau(mn) = \tau(m) \tau(n)$ numbers $d_i e_{i'}$ where $1 \leq i \leq \tau(m)$ and $1 \leq i' \leq \tau(n)$.
	However, we need the divisors of $mn$ to be listed in some specific linear order since such an ordering
	will allow us to label the corresponding superclasses and supercharacters on $\Z/mn\Z$.
	We therefore impose the lexicographic order on the set
	\begin{equation*}
		\{ d_i e_{i'} : 1 \leq i \leq \tau(m),\,\, 1 \leq i' \leq \tau(n) \}
	\end{equation*}
	of divisors of $mn$ which is induced by the lexicographic ordering of the Cartesian product
	$\{1,2,\ldots,\tau(m)\}\times \{1,2,\ldots,\tau(n)\}$.  This gives rise to an order-preserving bijection
	\begin{equation*}
		\sigma:\{1,2,\ldots,\tau(m)\}\times \{1,2,\ldots,\tau(n)\} \to \{1,2,\ldots,\tau(mn)\}
	\end{equation*}
	which implicitly provides us with an ordered list of the divisors of $mn$.  We now
	consider the matrices $M_1(mn),M_2(mn),\ldots,M_{\tau(mn)}(mn)$, the diagonal matrices
	$D_1(mn),D_2(mn),\ldots,D_{\tau(mn)}(mn)$, and the matrix $W(mn)$ which intertwines them.
	
	\begin{Lemma}\label{LemmaMegaTensor}
		Maintaining the notation and conventions described above,
		\begin{align}
			[M_{\sigma(i,i')}(mn)]_{\sigma(j,j'), \sigma(k,k')} &= [M_i(m)]_{j,k} [M_{i'}(n)]_{j',k'}, \label{eq-TensorM}\\[4pt]
			[D_{\sigma(i,i')}(mn)]_{\sigma(j,j'), \sigma(k,k')} &= [D_i(m)]_{j,k} [D_{i'}(n)]_{j',k'}, \label{eq-TensorD}\\[4pt]
			[W(mn)]_{\sigma(j,j'), \sigma(k,k')} &= [W(m)]_{j,k} [W(n)]_{j',k'},\label{eq-TensorW}
		\end{align}
		hold for all $1\leq i,j,k \leq \tau(m)$ and $1 \leq i',j',k' \leq \tau(n)$.
		In other words, the simultaneous diagonalization \eqref{eq-Similarity} of Theorem \ref{TheoremPowerful}
		is compatible with Kronecker products in the sense that
		\begin{equation*}
			\big(\underbrace{M_i(m) \otimes M_{i'}(n)}_{M_{\sigma(i,i')}(mn)} \big)
			\big(\underbrace{W(m) \otimes W(n)}_{W(mn)} \big)
			= \big(\underbrace{W(m) \otimes W(n)}_{W(mn)} \big)
			\big( \underbrace{D_i(m)\otimes D_{i'}(n)  }_{D_{\sigma(i,i')(mn)}} \big)
		\end{equation*}
		whenever $(m,n)=1$.
		In particular, $W(mn) = S(mn)$, the corresponding supercharacter table for $\Z/mn\Z$.
	\end{Lemma}
	
	\begin{proof}
		The given lists $d_1,d_2,\ldots,d_{\tau(m)}$ and $e_1,e_2,\ldots,e_{\tau(n)}$ provide us with corresponding superclasses
		\begin{equation*}
			K_i(m) = \Big\{ a \in \Z/m\Z : (a,m) = \frac{m}{d_i} \Big\},\qquad
			K_{i'}(n) = \Big\{ b \in \Z/n\Z : (b,n) = \frac{n}{e_{i'}} \Big\},
		\end{equation*}
		in $\Z/m\Z$ and $\Z/n\Z$.  
		For each pair $(i,i')$ satisfying $1 \leq i \leq \tau(m)$ and $1 \leq i' \leq \tau(n)$, we define
		\begin{equation*}
			K_{\sigma(i,i')}(mn) = \Big\{ c \in \Z/mn\Z : (c,mn) = \frac{mn}{d_i e_{i'}} \Big\},
		\end{equation*}
		yielding a partition of $\Z/mn\Z$ into $\tau(m)\tau(n) = \tau(mn)$ superclasses.
		By the Chinese Remainder Theorem, the map $\Phi: \Z/m\Z\times \Z/n\Z\to \Z/mn\Z$ defined by
		$\Phi\big( (a,b) \big) = ab\, (\operatorname{mod} mn)$ is a ring isomorphism which satisfies
		\begin{equation}\label{eq-PKK}
			\Phi\big( K_i(m) \times K_{i'}(n) \big) = K_{\sigma(i,i')}(mn).
		\end{equation}
	
		Fix elements $z,z'$ in $K_k(m)$ and $K_{k'}(n)$, respectively, and let
		\begin{align*}
			a_{i,j,k}(m) &=|\{ (x,y) \in K_i(m) \times K_j(m) : x + y = z \}|,\\
			a_{i',j',k'}(n) &=|\{ (x',y') \in K_{i'}(n) \times K_{j'}(n) : x' + y' = z' \}|.
		\end{align*}
		Clearly the product $a_{i,j,k}(m)a_{i',j',k'}(n)$ equals the number of solutions to
		\begin{equation*}
			(x,x') + (y,y') = (z,z'),
		\end{equation*}
		where $(x,x')$ belongs to $K_i(m) \times K_{i'}(n)$ and $(y,y')$ belongs to $K_j(m) \times K_{j'}(n)$.
		Applying $\Phi$ to both sides of the preceding and using \eqref{eq-PKK}, it follows that
		$a_{i,j,k}(m)a_{i',j',k'}(n)$ equals the number of solutions to $X+Y = Z$ where
		$Z = \Phi(z,z')$ is a fixed element of $K_{\sigma(k,k')}(mn)$, $X$ belongs to $K_{\sigma(i,i')}(mn)$,
		and $Y$ belongs to $K_{\sigma(j,j')}(mn)$, respectively.  In other words,
		\begin{equation}\label{eq-AM}
			a_{\sigma(i,i'), \sigma(j,j'), \sigma(k,k')}(mn) = a_{i,j,k}(m)a_{i',j',k'}(n),
		\end{equation}
		so that
		\begin{equation*}
			[M_{\sigma(i,i')}(mn)]_{\sigma(j,j'), \sigma(k,k')} = [M_i(m)]_{j,k} [M_{i'}(n)]_{j',k'}.
		\end{equation*}
		This establishes \eqref{eq-TensorM}.
		
		Turning our attention to \eqref{eq-TensorD}, we recall from Theorem \ref{TheoremPowerful} that 
		the matrix $D_{\sigma(i,i')}(mn)$ is diagonal.  In fact, its diagonal entries are precisely
		\begin{align*}
			[D_{\sigma(i,i')}(mn)]_{\sigma(j,j'), \sigma(j,j')}
			&= \frac{|K_{\sigma(i,i')}(mn)|   c_{d_j e_{j'}} (\frac{mn}{d_i e_{i'}}) }{|K_{\sigma(j,j')}(mn)|}  && \text{by \eqref{eq-OmegaDefinition}, \eqref{eq-SIDF}} \\
			&= \frac{  |K_i(m)||K_{i'}(n)| c_{d_j}( \frac{m}{d_i}) c_{e_{j'}}(\frac{n}{e_{i'}})}{ |K_j(m)||K_{j'}(n)| } && \text{by \eqref{eq-PKK}, \eqref{eq-cmnxy}}\\
			&= \frac{ |K_i(m)| c_{d_j}(\frac{m}{d_i})}{|K_j(m)|} \cdot \frac{ |K_{i'}(n)| c_{e_{j'}}(\frac{n}{d_{i'}})}{|K_{j'}(n)|} \\
			&= [D_i(m)]_{j,j} [D_{i'}(n)]_{j',j'}.&& \text{by \eqref{eq-OmegaDefinition}}
		\end{align*}
		The proof of \eqref{eq-TensorW} is similar, and in fact much easier, since $W(m)$, $W(n)$, and $W(mn)$ do not depend
		upon the indices $i,i'$.  The fact that $W(mn) = S(mn)$ follows immediately from
		\eqref{eq-WPASPA} and \eqref{eq-Tensor}.
	\end{proof}

	Our primary interest in the preceding lemma lies in the following straightforward generalization.
	Let $n=p_1^{\alpha_1} p_2^{\alpha_2} \cdots p_r^{\alpha_r}$ denote the factorization of $n$
	into distinct primes and let $d_1,d_2,\ldots,d_{\tau(n)}$ denote the divisors of $n$.  Noting that
	\begin{equation*}
		\tau(n) = \prod_{\ell=1}^r (\alpha_{\ell}+1),
	\end{equation*}
	we let
	\begin{equation*}
		\sigma: \prod_{\ell=1}^r \{1,2,\ldots,\alpha_{\ell}+1\} \to \{1,2,\ldots,\tau(n)\}
	\end{equation*}
	be a bijection which preserves the lexicographic order on the Cartesian product $\prod_{\ell=1}^r \{1,2,\ldots,\alpha_{\ell}+1\}$.
	According to this labeling scheme,
	\begin{equation}\label{eq-DCS}
		d_{\sigma(i_1,i_2,\ldots,i_r)}=p_1^{i_1-1} p_2^{i_2-1}\cdots p_r^{i_r-1}
	\end{equation}
	is the $\sigma(i_1,i_2,\ldots,i_r)$th divisor of $n$.  In light of Lemma \ref{LemmaMegaTensor}, it follows that
	\begin{equation}\label{eq-TBC}
		D_{\sigma(i_1,i_2,\ldots,i_r)}(n) \cong
		\bigotimes_{\ell=1}^r D_{i_{\ell}}(p_{\ell}^{\alpha_{\ell}} )  \sim
		\bigotimes_{\ell=1}^r M_{i_{\ell}}(p_{\ell}^{\alpha_{\ell}}) \cong
		M_{\sigma(i_1,i_2,\ldots,i_r)}(n) ,
	\end{equation}
	where $\sim$ denotes similarity and $\cong$ denotes similarity by a permutation matrix.
	In particular, the eigenvalues of the diagonal matrix $D_{\sigma(i_1,i_2,\ldots,i_r)}(n)$
	are precisely the $\tau(n)$ numbers
	\begin{equation*}
		c_{d_{\sigma(i_1,i_2,\ldots,i_r)}}(d_j)
	\end{equation*}
	for $j=1,2,\ldots,\tau(n)$.  We can therefore obtain from \eqref{eq-TBC} a variety of formulas involving Ramanujan sums 
	by utilizing the detailed information about the matrices $M_{i_{\ell}}(p_{\ell}^{\alpha_{\ell}})$ we derived in 
	Subsection \ref{SubsectionMPA}.
	 
\subsection{Ramanujan sums and superclass constants}	 
	Fix a positive integer $n$
	and let $a_{i,j,k}=a_{i,j,k}(n)$, $M_i = M_i(n)$, and so forth.
	Now let us observe that
	\begin{align*}
	M_i 
	&= WD_iW^{-1} && \text{by Theorem \ref{TheoremPowerful}}\\
	&= S D_i S^{-1} && \text{by Lemma \ref{LemmaMegaTensor}}\\
	&= \frac{1}{n} SD_iS. && \text{by Theorem \ref{TheoremSU}}
	\end{align*}
	Comparing the $(j,k)$ entry in the equality $nM_i = SD_iS$ yields the formula
	\begin{equation*}
		\boxed{n a_{i,j,k} = \sum_{d|n} c_{d_i}(d) c_{d_j}(d) c_{\frac{n}{d}} \Big( \frac{n}{d_k} \Big).}
	\end{equation*}
	In light of \eqref{eq-Reciprocity}, we may rewrite the preceding in the more symmetric form
	\begin{equation}\label{eq-TSCC}
		\boxed{ n \phi(d_k) a_{i,j,k} = \sum_{d|n} \phi\left( \frac{n}{d} \right) c_{d_i}(d) c_{d_j}(d) c_{d_k}(d).}
	\end{equation}
	Since the right side of \eqref{eq-TSCC} is symmetric in $i,j,k$, it follows
	that $\phi(d_k) a_{i,j,k} = \phi(d_j) a_{k,i,j}$.  By definition, the superclass constants satisfy 
	$a_{k,i,j} = a_{i,k,j}$, whence
	\begin{equation}
		\phi(d_k) a_{i,j,k} = \phi(d_j) a_{i,k,j}.
	\end{equation}
	Indeed, this pattern is evident in the structure \eqref{eq-MPN} of the matrices
	$M_i(p^{\alpha})$.  Another consequence of 
	\eqref{eq-TSCC} is the following result.
	
	\begin{Theorem}
		If $n = p_1^{\alpha_1} p_2^{\alpha_2}\cdots p_r^{\alpha_r}$ is the canonical factorization
		of $n$ into distinct primes $p_1,p_2,\ldots, p_r$ and $d = p_1^{\beta_1} p_2^{\beta_2} \cdots p_r^{\beta_r}$
		is a divisor of $n$, then
		\begin{equation}\label{eq-Cubic}
		\boxed{
			\sum_{d'|n} \phi\left(\frac{n}{d'}\right) \big(c_d(d')\big)^3 =
			\begin{cases}
			0 & \text{if $2|d$},\\[4pt]
			n\phi(d) \displaystyle\prod_{\ell=1}^r  \big\lceil  p_{\ell}^{\beta_{\ell}} - 2 p_{\ell}^{\beta_{\ell}-1} \big\rceil
			& \text{if $2\nmid d$},
			\end{cases}
		}
		\end{equation}
		where $\lceil\,\cdot\,\rceil$ denotes the ceiling function.
	\end{Theorem}
	
\begin{proof}
By Corollary \ref{CorollarySupermultiplicative} and the multiplicativity of the Euler totient function,
it suffices to establish the desired identity when $n=p^{\alpha}$ is a prime power.
Let $d=d_i = p^{i-1}$ and recall from \eqref{eq-MPN} 
that $a_{i,i,i}(p^{\alpha}) = 1$ if $i = 1$ and $a_{i,i,i}(p^{\alpha}) = p^{i-1}-2p^{i-2}$ otherwise.
Setting $i=j=k$ in \eqref{eq-TSCC}, we obtain \eqref{eq-Cubic} for $n= p^{\alpha}$.
\end{proof}

\subsection{Power sum identities}\label{SubsectionPower}
	The preceding material now allows us to rapidly produce a wide variety of power sum identities 
	involving Ramanujan sums, all of which appear to be novel.  This highlights our larger argument, namely that
	supercharacter theory and superclass arithmetic are powerful tools which can yield new insights,
	even when applied to the most elementary of groups (i.e., the cyclic groups $\Z/n\Z$).

	\begin{Theorem}\label{TheoremPositivePowers}
		If $n = p_1^{\alpha_1} p_2^{\alpha_2}\cdots p_r^{\alpha_r}$ is the canonical factorization
		of $n$ into distinct primes $p_1,p_2,\ldots, p_r$ and $d = p_1^{\beta_1} p_2^{\beta_2} \cdots p_r^{\beta_r}$
		is a divisor of $n$, then for $s=0,1,2,\ldots$ we have
		\begin{equation*}
			\boxed{
				\sum_{k|n} \big(c_d(k) \big)^s
				= \prod_{\ell=1}^{r} \Big( (\alpha_{\ell} - \beta_{\ell} +1 )
				\phi(p_{\ell}^{\beta_{\ell}})^s + (-1)^s \lfloor p_{\ell}^{\beta_{\ell}-1} \rfloor^s \Big)	
			}
		\end{equation*}	
		Here $\lfloor \,\cdot\, \rfloor$ denotes the floor function.
	\end{Theorem}
	
	\begin{proof}
		By Corollary \ref{CorollarySupermultiplicative} it suffices to establish the desired formula
		when $n = p^{\alpha}$ is a prime power.	
		Recall from \eqref{eq-MPN} that for $2\leq i \leq \alpha+1$ we have
		\begin{equation}\label{eq-MPBC}
			M_i(p^{\alpha}) = \minimatrix{0}{\vec{b}}{\vec{c}}{x} \oplus \phi(p^{i-1}) I_{\alpha-i+1}
		\end{equation}
		where $\vec{b}$ and $\vec{c}$ are certain $(i-1)\times 1$ and $1 \times (i-1)$ matrices which satisfy
		\begin{equation*}
			\vec{c} \vec{b} = \phi(p^{i-1})\big( 1 + (p-1) + (p^2-p) + \cdots + (p^{i-2} - p^{i-3} \big) = \phi(p^{i-1})p^{i-2}
		\end{equation*}
		and where $x = p^{i-1} - 2p^{i-2}$.  
		A short inductive argument confirms that for each $s =0, 1,2,\ldots$ there exist corresponding
		polynomials $f_{11},f_{12},f_{21},f_{22}$ such that
		\begin{equation}\label{eq-bcx}
			\minimatrix{0}{\vec{b}}{\vec{c}}{x}^s 
			= \minimatrix{ f_{11}(\vec{b}\vec{c}) }{ f_{12}(\vec{b}\vec{c})\vec{b}}{ \vec{c} f_{21}(\vec{b}\vec{c}) }{ f_{22}(\vec{c}\vec{b})}.
		\end{equation}
		In particular, note that the preceding formula also holds when $\vec{b}$ and $\vec{c}$ are replaced by positive
		real numbers $b$ and $c$ and when the matrices involved are regarded simply as $2 \times 2$ matrices with real entries.
		Letting $b,c > 0$ satisfy
		\begin{equation}\label{eq-cbxcbx}
			cb = \vec{c}\vec{b}= \phi(p^{i-1})p^{i-2},
		\end{equation}
		it follows from \eqref{eq-bcx} and an elementary diagonalization argument that
		\begin{align*}
			\tr \minimatrix{0}{\vec{b}}{\vec{c}}{x}^s 
			&= \tr f_{11}(\vec{b}\vec{c}) + \tr f_{22}(\vec{c}\vec{b}) \\
			&= \tr f_{11}(\vec{c}\vec{b}) + \tr f_{22}(\vec{c}\vec{b}) \\
			&= \tr f_{11}(cb) + \tr f_{22}(cb) \\
			&= \tr \minimatrix{0}{b}{c}{x}^s \\
			&= \left( \frac{x + \sqrt{4bc+x^2}}{2} \right)^s + \left( \frac{x - \sqrt{4bc+x^2}}{2} \right)^s  \\
			&= \left[ \frac{(p^{i-1} - 2p^{i-2}) + p^{i-1}}{2} \right]^s + \left[ \frac{(p^{i-1} - 2p^{i-2}) - p^{i-1}}{2} \right]^s  \\
			&= \left( p^{i-1}-p^{i-2} \right)^s + \left( -p^{i-2}\right)^s  \\
			&= \phi( p^{i-1})^s + \left( -p^{i-2}\right)^s.		
		\end{align*}
		Returning to \eqref{eq-MPBC}, we find that
		\begin{equation}\label{eq-MIPSPS}
			\tr M_i^s(p^{\alpha}) =  (\alpha - i + 2)\phi( p^{i-1})^s + (-p^{i-2})^s
		\end{equation}
		when $2 \leq i \leq \alpha+1$.
		Since $M_1(p^{\alpha})$ is the $(\alpha+1)\times (\alpha+1)$ identity matrix, setting $\beta = i-1$ it follows that
		\begin{align*}
			\sum_{k|p^{\alpha}} \big( c_{p^{\beta}}(k) \big)^s
			&= \tr D_{\beta+1}^s(p^{\alpha}) && \text{by \eqref{eq-PrimePowerD}} \\[-8pt]
			&= \tr M_{\beta+1}^s(p^{\alpha}) && \text{by Theorem \ref{TheoremPowerful}}\\
			&=(\alpha - \beta + 1)\phi( p^{\beta})^s + (-1)^s \lfloor p^{\beta-1}\rfloor^s, && \text{by \eqref{eq-MIPSPS}}
		\end{align*}
		as required.
	\end{proof}
	
	In light of \eqref{eq-Commute}, it is not hard to generate more complicated variants of the preceding formula.
	For instance, since
	\begin{equation*}
		\tr M_i^s(p^{\alpha}) M_j^s(p^{\alpha}) =
		\begin{cases}
			\phi(p^{i-1})^s \tr M_j^s(p^{\alpha}) & \text{if $i < j$}, \\[5pt]
			\tr M_i^{2s}(p^{\alpha}) & \text{if $i = j$},\\[5pt]
			\phi(p^{j-1})^s \tr M_i^s(p^{\alpha}) & \text{if $i > j$},
		\end{cases}
	\end{equation*}
	the quantity $\tr M_i^s(p^{\alpha}) M_j^s(p^{\alpha})$ can be evaluated using similar methods.
	A little algebra then yields the following generalization of Theorem \ref{TheoremPositivePowers}.
	
	\begin{Theorem}
		Let $n = p_1^{\alpha_1} p_2^{\alpha_2}\cdots p_r^{\alpha_r}$ be the canonical factorization
		of $n$ into distinct primes $p_1,p_2,\ldots, p_r$.  If $d = p_1^{\beta_1} p_2^{\beta_2} \cdots p_r^{\beta_r}$ and
		$d' = p_1^{\gamma_1} p_2^{\gamma_2} \cdots p_r^{\gamma_r}$
		are divisors of $n$, then for $s=0,1,2,\ldots$ we have
		\begin{align*}
			&\sum_{k|n} \big(c_d(k) c_{d'}(k)\big)^s\\
			&\quad= \prod_{\ell=1}^{r} 
			\Big( (\alpha_{\ell} - \max\{\beta_{\ell},\gamma_{\ell}\} + 1)
			\phi(p_{\ell}^{\beta_{\ell}})^s \phi(p_{\ell}^{\gamma_{\ell}})^s  \\
			&\qquad \qquad+ 
			\big( \lfloor p^{\min\{\beta_{\ell},\gamma_{\ell}\}-1} \rfloor - (1-\delta_{\beta_{\ell},\gamma_{\ell}})
			p^{\min\{\beta_{\ell},\gamma_{\ell}\}} \big)^s
			\lfloor p^{ \max\{ \beta_{\ell}, \gamma_{\ell}\}-1} \rfloor^s
			\Big),
		\end{align*}	
		where $\delta$ denotes the Kronecker delta function.
	\end{Theorem}

	Although it might at first appear that the preceding results could be adapted to handle negative
	exponents $s$, there are a few minor obstacles.  First is the fact that $M_i(p^{\alpha})$ is invertible if and only if $i = 1$ or $i = 2$.
	This corresponds to the fact that $c_d(k)$ vanishes for certain values of $k$ if $d$ is not square-free.  Indeed, this
	can be seen directly from von Sterneck's formula \eqref{eq-vonSterneck} and the definition 
	\eqref{eq-PrimePowerMobiusPhi} of the M\"obius $\mu$-function.  The correct adaptation of
	Theorem \ref{TheoremPositivePowers} for negative exponents is the following.

	\begin{Theorem}
		If $n = p_1^{\alpha_1} p_2^{\alpha_2}\cdots p_r^{\alpha_r}$ is the canonical factorization
		of $n$ into distinct primes $p_1,p_2,\ldots, p_r$ and $d = p_1^{\beta_1} p_2^{\beta_2} \cdots p_r^{\beta_r}$
		is a square-free divisor of $n$ (i.e., $0 \leq \beta_i \leq 1$ for $i=1,2,\ldots,r$), then for $s =0,1,2,\ldots$ we have
		\begin{equation*}
			\boxed{
			\sum_{k|n} \frac{1}{\big(c_{d}(k) \big)^s} 
			= \prod_{\ell=1}^r \left( \frac{\alpha_{\ell}}{ (p_{\ell}-1)^{\beta_{\ell} s} } + (-1)^{\beta_{\ell}s} \right).
			}
		\end{equation*}
	\end{Theorem}

	\begin{proof}
	As before, it suffices to prove the desired formula when $n = p^{\alpha}$ is a prime power.
	The upper-left $2 \times 2$ submatrix of the $(\alpha+1) \times (\alpha+1)$ matrix
	\begin{equation*}
		M_2(p^{\alpha}) = \small
		\left[
		\begin{array}{cc|cccc}
			0 & 1 & \0 &  \0 & \gcdots & \0 \\
			p-1 & p-2 &  \0 &  \0 & \gcdots & \0 \\
			\hline
			\0 & \0 & p-1 & \0 & \gcdots & \0 \\
			\0 & \0 & \0 & p-1 & \gcdots & \0 \\
			\gvdots & \gvdots &  \gvdots & \gvdots & \gddots & \gvdots \\
			\0 & \0 & \0 & \0 & \gcdots & p-1 \\
		\end{array}
		\right]
		\end{equation*}
		has the eigenvalues $-1$ and $p-1$.  On the other hand, $M_1(p^{\alpha})$ is the identity matrix
		and hence has the eigenvalue $1$ with multiplicity $\alpha+1$.  Therefore
		\begin{align*}
			\sum_{k|p^{\alpha}} \frac{1}{( c_{p^{\beta}}(k) \big)^s}
			&= \tr D_{\beta+1}^{-s}(p^{\alpha}) && \text{by \eqref{eq-PrimePowerD}} \\[-8pt]
			&= \tr M_{\beta+1}^{-s}(p^{\alpha}) && \text{by Theorem \ref{TheoremPowerful}}\\
			&= \frac{\alpha}{ (p-1)^{\beta s} } + (-1)^{\beta s},
		\end{align*}
		as required.
	\end{proof}
	
	Along similar lines, we have the following.
	
	\begin{Theorem}
			If $n = p_1^{\alpha_1} p_2^{\alpha_2}\cdots p_r^{\alpha_r}$ is the canonical factorization
			of $n$ into distinct primes $p_1,p_2,\ldots, p_r$ and $d = p_1^{\beta_1} p_2^{\beta_2} \cdots p_r^{\beta_r}$
			is a square-free divisor of $n$ (i.e., $0 \leq \beta_i \leq 1$ for $i=1,2,\ldots,r$), then for $z$ in $\C$ we have
			\begin{equation}\label{eq-Zeta}
				\boxed{
				\sum_{k|n} \frac{1}{|c_{d}(k) |^z} 
				= \prod_{\ell=1}^r \left(1+  \frac{\alpha_{\ell}}{ (p_{\ell}-1)^{\beta_{\ell} z} } \right).
				}
			\end{equation}
	\end{Theorem}
	
	\begin{proof}
		Noting that $|c_d(k)|^z = (c_d(k)^2)^{\frac{z}{2}}$ and that
		\begin{equation*}
			\tr \big(M^2_{\beta+1}(p^{\alpha}) \big)^{-\frac{z}{2}} = 
			1+ \frac{\alpha}{ (p-1)^{\beta z} },
		\end{equation*}
		the proof is similar to the proof of Theorem \ref{TheoremPositivePowers}.
	\end{proof}
	
	Since the expression \eqref{eq-Zeta} bears some resemblance to the classical
	Riemann $\zeta$-function and its variants, it is natural to ask whether this expression
	obeys the analogue of the Riemann Hypothesis.  The following result shows that this occurs
	if and only if $n$ and $d$ satisfy some rather peculiar hypotheses.
	
	\begin{Corollary}
		The complex roots of the function \eqref{eq-Zeta} all lie on the line $\Re z = \frac{1}{2}$
		if and only if each prime $p$ which divides $d$ is of the form $\alpha^2+1$ where
		$p^{\alpha}$ is the highest power of $p$ which divides $n$.
	\end{Corollary}

	\begin{proof}
		The complex roots of \eqref{eq-Zeta} are precisely the numbers		
		\begin{equation*}
			\qquad\qquad z = \frac{ \log|\alpha_{\ell}| + i (2m+1)\pi}{\log (p_{\ell}-1)},\qquad m\in\Z,
		\end{equation*}
		for those $\ell$ such that $\beta_{\ell} = 1$ (i.e., for those primes $p_{\ell}$ which divide $d$).  
		The real part of the preceding clearly equals $\frac{1}{2}$
		if and only if $p_{\ell}=\alpha_{\ell}^2+1$.
	\end{proof}

	\begin{Example}
		Let $n = 5^2 \times 17^4 \times 37^6 = 5,357,300,885,152,225$ and $d = 5\times 17\times 37 = 3,145$.
		Since $5 = 2^2+1$, $17 = 4^2 + 1$, and $37 = 6^2+1$, it follows that the
		corresponding ``$\zeta$-function'' \eqref{eq-Zeta} satisfies the Riemann Hypothesis.
	\end{Example}

	Using \eqref{eq-Commute} and some of the preceding computations, it is not hard to explicitly 
	evaluate the trace of any word composed 
	using $M_1(p^{\alpha}), M_2(p^{\alpha}),\ldots,M_{\alpha+1}(p^{\alpha})$ and $M_2(p^{\alpha})^{-1}$.
	Consequently, the motivated individual could in principle provide an explicit formula for the sum
	\begin{equation*}
		\sum_{k|n} \frac{ c_{f_1}(k)^{s_1} c_{f_2}(k)^{s_2} \cdots c_{f_{\eta}}(k)^{s_{\eta}} }
		{ c_{g_1}(k)^{t_1} c_{g_2}(k)^{t_2} \cdots c_{g_{\nu}}(k)^{t_{\nu}} },
	\end{equation*}
	where $f_1,f_2,\ldots,f_{\eta}$ are divisors of $n$ and $g_1,g_2,\ldots,g_{\nu}$ are square-free
	divisors of $n$.  We make no attempt to do so here, having made our point that
	the arithmetic of superclasses can be used to deduce a variety of identities for
	Ramanujan sums.

\section{Conclusion}

	We have demonstrated that reexamining even the most elementary of groups, namely the cyclic groups $\Z/n\Z$,
	from the perspective of supercharacter theory can yield surprising results.  In particular, almost the entire 
	algebraic theory of Ramanujan sums can be derived, in a systematic manner, using this approach.
	Many of the familiar classical identities for these fascinating sums, along with a variety of new ones,
	can be obtained with minimal effort once the basic machinery has been developed.
	
	All of our results flow directly from a general theoretical framework without ad hoc arguments.
	More importantly, many of the ideas developed in this note can be applied to arbitrary finite groups.
	In particular, the arithmetic of superclasses (Section \ref{SectionSuperclass}) and the
	simultaneous diagonalization theorem (Theorem \ref{TheoremPowerful}), which yielded
	some of the more elaborate identities for Ramanujan sums, hold in much greater generality.
	We therefore hope that revisiting other families of elementary groups (e.g., dihedral groups, 
	Frobenius groups, symmetric groups, $p$-groups,\ldots)
	from the perspective of supercharacter theory might yield further information about other exponential sums
	(e.g., Gauss sums, Kloosterman sums, Jacobi sums, and their variants) which are of interest
	in number theory (see \cite{SESUP}).

\appendix

\section{Computing the matrix $M_i(p^{\alpha})$}\label{SectionPrimePowerAlpha}
	
	This appendix contains a detailed derivation of the description \eqref{eq-MPN} for the
	matrix $M_i(p^{\alpha})$ given in  Subsection \ref{SubsectionMPA}.
	
	The proof of Lemma \ref{LemmaPreliminary} tells us that $a_{i,j,k}$ is independent of the particular
	representative $z$ of $K_k$ which is chosen.  Since $\Z/p^{\alpha}\Z$ is abelian, we also note that
	\begin{equation}\label{eq-ijk}
		a_{i,j,k} = a_{j,i,k}
	\end{equation}
	and
	\begin{equation}\label{eq-ijk2}
		a_{i,j,k} =0 \quad\iff\quad a_{i,k,j}=0.
	\end{equation}
	Statement \eqref{eq-ijk2} requires some explanation.  Observe that
	$a_{i,j,k} = 0$ holds if and only if $x+y=z$ has no solutions $(x,y,z)$ in $K_i \times K_j \times K_k$.
	Since $K_j = -K_j$ and $K_k = -K_k$ by \eqref{eq-KiDFN}, it follows that the preceding happens if and only if
	$x + z'= y'$ has no solutions $(x,z',y')$ in $K_i \times K_k \times K_j$.
	On the other hand, it is important to note that $a_{i,j,k} = a_{i,k,j}$ does not hold in general
	since the fixed representative $z$ of $K_k$ used in the equation \eqref{eq-xyz} plays a distinguished role.

	We first break down the evaluation of the $a_{i,j,k}$ into five special cases, from which the 
	structure of the matrix $M_i = M_i(p^{\alpha})$ can eventually be deduced.
	
	\begin{Lemma}\label{LemmaFive}
		For $G = \Z/p^{\alpha}\Z$ and $K_i = \{ xp^{\alpha-i+1} \in \Z/p^{\alpha}\Z : p \nmid x\}$, we have
		\begin{enumerate}[(a)]\addtolength{\itemsep}{0.5\baselineskip}
			\item if $k > i$ and $j \neq k$, then $a_{i,j,k} = 0$,
			\item if $j=k>i$, then $a_{i,j,k} = \phi(p^{i-1})$,
			\item if $j=k=i$, then $a_{i,j,k} = p^{i-1} - 2p^{i-2}$,
			\item if $j>k$ and $i\neq j$, then $a_{i,j,k} = 0$,
			\item if $i=j>k$, then $a_{i,j,k} = \phi(p^{i-1})$.
		\end{enumerate}
	\end{Lemma}

	\begin{proof}
		We first prove (a).  Letting $k > i$ and $j \neq k$, we may assume that $i \leq j$ by \eqref{eq-ijk}.  
		If $x_i = xp^{\alpha-i+1}$ and $y_j = yp^{\alpha-j+1}$ belong to $K_i$ and $K_j$, respectively, then it follows that
		\begin{equation}\label{eq-PCXIYJ}
			x_i + y_j = p^{\alpha-j+1}(xp^{j-i} + y)
		\end{equation}
		belongs to $K_j$ since $xp^{j-i}+y$ is not divisible by $p$ (recall that $p\nmid x$ and $p\nmid y$ by definition of 
		$K_i$ and $K_j$).  Since $j \neq k$, it follows from \eqref{eq-PCXIYJ}
		that $x_i+y_j$ cannot belong to $K_k$ from which it follows that  $a_{i,j,k} = 0$.
		
		Next we consider (b).  Suppose that $j = k > i$ and fix $z_k = zp^{\alpha-k+1}$ in $K_k$.
		Since $j=k$, the computation \eqref{eq-PCXIYJ} tells us that for each $x_i$ in $K_i$, there
		exists a unique $y_j = z-xp^{j-i}$ such that $x_i + y_j = z_k$.  Thus $a_{i,j,k} = |K_i| = \phi(p^{i-1})$.
		
		The proof of (c) is somewhat more involved.  Let $i=j=k$ and fix $z_k = zp^{\alpha-k+1}$ where $p \nmid z$.
		For any element $x_i = xp^{n-k+1}$ of $K_i = K_k$, there exists a unique $a_0$ in $\Z/p^{\alpha} \Z$ such that
		\begin{equation}\label{eq-XIAO}
			x_i + a_0 = z_k.
		\end{equation}
		Since $p^{\alpha-k+1}$ divides both $x_i$ and $z_k$, it must also divide $a_0$ so that we can
		write $a_0  = ap^{n-k+1}$ for some $a$.  In light of \eqref{eq-XIAO}, we now have
		$x+a=z$ so that $a = z-x$.  Therefore $a_0$ belongs to $K_k$ if and only if $p \nmid (z-x)$.
		We now note that if $p|(z-x)$, then $x_i$ would serve as a solution to $x_i + y = z_k$
		where $y$ belongs to some $K_{\ell}$ with $\ell < i=k$.  By statement (b), it follows that
		\begin{align*}
			a_{i,j,k}
			&= |K_k| - \sum_{\ell=1}^{k-1} \phi(p^{\ell-1}) \\
			&= \phi(p^k) - \sum_{\ell=2}^{k-1} (p^{\ell-1}-p^{\ell-2}) -1\\
			&=(p^{k-1} - p^{k-2})- (p^{k-2} - p^{k-3}) - \cdots - (p-1) -1\\
			&= p^{k-1} - 2 p^{k-2},
		\end{align*}
		as claimed.
		
		Now we consider statement (d).  Suppose that $j > k$ and $i \neq j$.  In light of \eqref{eq-ijk}, we
		may assume that $i < j$.  Maintaining the same notation and conventions as in the proof of statement
		(a), we again arrive at the equation \eqref{eq-PCXIYJ} and conclude that $x_i + y_j$ belongs to $K_j$.
		Since $j > k$, we conclude that $a_{i,j,k} = 0$.
		
		Finally, let us prove (e).  Suppose that $i = j > k$ and let $z_k = zp^{\alpha-k+1}$ in $K_k$ be given.  For each
		$x_i = xp^{\alpha-i+1}$ in $K_i$, we have
		\begin{equation}\label{eq-ZDA}
			xp^{\alpha-i+1} + (zp^{i-k} - x)p^{\alpha-i+1} = (zp^{i-k})p^{\alpha-i+1} = zp^{\alpha-k+1} = z_k.
		\end{equation}
		Now $p|(zp^{i-k})$ because $i > k$, so it follows that $p \nmid(zp^{i-k} -x)$ since $p \nmid x$.
		Therefore $(zp^{i-k} - x)p^{\alpha-i+1}$ belongs to $K_i$.
		Looking at \eqref{eq-ZDA} we conclude that for each $x_i$ in $K_i$, there exists
		a unique $y_i$ in $K_i$ such that $x_i + y_i = z_k$.  Since $i = j$ we conclude that
		$a_{i,j,k} = |K_i| = \phi(p^{i-1})$, as desired.
	\end{proof}

	Having proven the preceding lemma, it is now straightforward to see that $M_i$ has the form
	\eqref{eq-MPN}.  The complete reasoning is presented below.
	\begin{enumerate}\addtolength{\itemsep}{0.5\baselineskip}
		\item If $j,k < i$, then $(M_i)_{j,k} = 0$.  In other words, the upper-left $(i-1) \times (i-1)$ submatrix of $M_i$ contains
			only zeros.  Indeed, letting $i'=j$ and $j' = i > k$ (so that $i'\neq j'$ since $j<i$), it follows that		
			$(M_i)_{j,k} = a_{i,j,k} = a_{j,i,k} = a_{i',j',k} = 0$ 
			by \eqref{eq-ijk} and (d) of Lemma \ref{LemmaFive}.
	
		\item If $j< i = k$, then $(M_i)_{j,k} = \phi(p^{j-1})$ so that the first $i-1$ entries of the $i$th column of $M_i$
			are given by $1, \phi(p),\phi(p^2),\ldots,\phi(p^{i-2})$.  As before, we set $i'=j$ and $j' = i$ and observe that
			$(M_i)_{j,k} = a_{i,j,k} = a_{j,i,k} = a_{i',j',k} = \phi(p^{i'-1}) = \phi(p^{j-1})$ by (b) of Lemma \ref{LemmaFive}
			since $j' = k > i'$.
			
		\item If $i= j > k$, then $(M_i)_{j,k} = \phi(p^{i-1})$ so that the first $(i-1)$ entries of the $i$th row of $M_i$ are each $\phi(p^{i-1})$.
			To see this, simply note that $(M_i)_{j,k} = a_{i,j,k} = \phi(p^{i-1})$ by (e) of Lemma \ref{LemmaFive}.
			
		\item If $i=j=k$, then $(M_i)_{j,k} = a_{i,i,i} = p^{i-1} - 2p^{i-2}$ by (c) of Lemma \ref{LemmaFive}.
		
		\item If $j = k > i$, then $(M_i)_{j,k} = \phi(p^{i-1})$.  In other words, the final $(n+1-i)$ entries along the main diagonal
			of $M_i$ are $\phi(p^{i-1})$.  This follows immediately from (b) of Lemma \ref{LemmaFive}.
			
		\item If $j > i,k$, then $(M_i)_{j,k} = 0$ follows from (d) of Lemma \ref{LemmaFive}.  Therefore
			the final $(n+1-i)$ rows of $M_i$ have only zeros to the left of the main diagonal.
			
		\item If $i,j < k$, then $(M_i)_{j,k} = 0$.  In other words, the last $(n+1-i)$ columns of $M_i$ have only zeros
			above the main diagonal.  This follows immediately from (a) of Lemma \ref{LemmaFive}.		
	\end{enumerate}

%
%
%

\bibliography{RSS}

\end{document}